\documentclass[a4paper,11pt]{amsart}

\usepackage{style}
\usepackage[marginratio=1:1,height=584pt,width=450pt,tmargin=117pt]{geometry}

\makeatletter
\newcases{centercases}{\quad}
  {\hfil$\m@th\displaystyle{##}$\hfil}
  {$\m@th\displaystyle{##}$\hfil}{\lbrace}{.}
\makeatother

\title[Cubic fourfolds with a symplectic automorphism of prime order]{Cubic fourfolds with a symplectic automorphism of prime order}

\author{Simone Billi}
\address[Simone Billi]{DIMA department of mathematics, University of Genova, Via dodecaneso 35, Cap 16146, Genova, Italy}
\email{simone.billi@edu.unige.it}

\author{Annalisa Grossi}
\address[Annalisa Grossi]{Alma Mater studiorum Università di Bologna Dipartimento di Matematica,
Piazza di Porta San Donato 5, Bologna, 40126 Italia}
\email{annalisa.grossi3@unibo.it}

\author{Lisa Marquand}
\address[Lisa Marquand]{Courant Institute of Mathematical Sciences, New York University, 251 Mercer St, NY 10012, USA}
\email{lisa.marquand@nyu.edu}

\date{\today}

\thanks{Simone Billi was partially supported by the Curiosity Driven 2021 Project "Varieties with trivial or
negative canonical bundle and the birational geometry of moduli spaces of curves: a
constructive approach" - Programma nazionale per la Ricerca (PNR) DM 737/2021.\\
Annalisa Grossi and Simone Billi were supported by the European Union - NextGenerationEU under the National Recovery and Resilience Plan (PNRR) - Mission 4 Education and research - Component 2 From research to business - Investment 1.1 Notice Prin 2022 - DD N. 104 del 2/2/2022, from title "Symplectic varieties: their interplay with Fano manifolds and derived categories", proposal code 2022PEKYBJ – CUP J53D23003840006. \\
Simone Billi and Annalisa Grossi are members of the INdAM group GNSAGA, and Simone Billi was partially supported by it.
Lisa Marquand was partially  supported by NSF grant DMS-2503390. }
\makeatletter
\@namedef{subjclassname@2020}{%
 \textup{2020} Mathematics Subject Classification}
\makeatother

\keywords{Cubic fourfolds, symplectic automorphism}

\thanks{}

\begin{document}

\begin{abstract}	
We determine the algebraic and transcendental lattices of a general cubic fourfold with a symplectic automorphism of prime order. We prove that cubic fourfolds admitting a symplectic automorphism of order at least three are rational, and we exhibit two families of rational cubic fourfolds that are not equivariantly rational with respect to their group of automorphisms. As an application, we determine the cohomological action of symplectic birational transformations of manifolds of OG10 type that are induced by prime order sympletic automorphisms of cubic fourfolds.
\end{abstract}

\maketitle
\section{Introduction}
In recent years, there has been renewed interest in studying automorphisms of cubic fourfolds from a Hodge theoretic perspective. The automorphisms of a cubic fourfold are linear, and those of prime order are classified in \cite{GonzalezAguilera.Liendo:aut.prime.ord.smooth.cubic.n-folds}, where an explicit description of the families of cubic fourfolds admitting such an automorphism is given (see also \cite{yu2020moduli}). If \(X\subset \mathbb{P}^5\) is a cubic fourfold, there is a Hodge decomposition \(H^4(X,\mathbb{C})=H^{3,1}(X)\oplus H^{2,2}(X)\oplus H^{1,3}(X)\) with \(h^{3,1}=1\). Much of the geometric information about the cubic fourfold is retained in the \textit{algebraic lattice} \(A(X)=H^4(X,\mathbb{Z})\cap H^{2,2}\) and in the \textit{transcendental lattice} \(T(X)=A(X)^\perp\). Every automorphism of a cubic fourfold induces an isometry on the \(H^4(X,\mathbb{Z})\), and automorphisms acting trivially on \( H^{3,1}(X)\) are called \textit{symplectic}, they are called \textit{non-symplectic} otherwise. 
A symplectic automorphism of a cubic fourfold $X$ induces a symplectic automorphism of the Fano variety of lines, $F(X)$, which is an irreducible holomorphic symplectic manifold \cite{Beauville.Donagi:variete.droites.hypersurf.cub.dim.4}.
Cubic fourfolds with a symplectic automorphism are studied in \cite{fu2016classification}, and groups of symplectic automorphisms are classified in \cite{Laza.Zheng:aut}. The algebraic and transcendental lattices of a cubic fourfold with an involution are determined in \cite{marquand2023cubic}, where the geometry of such cubic fourfolds is intensively studied. A similar approach is carried out in \cite{billigrossi2024non} to determine the algebraic and transcendental lattice, and describe the geometry of cubic fourfolds with a non-symplectic automorphism of higher order. A family of very symmetric cubic fourfolds with a symplectic automorphism of order three is considered in \cite{koike2022cubic}, where it is showed that such cubic fourfolds contain many planes that generate the algebraic lattice. 

The current paper determines the algebraic and transcendental lattices of a cubic fourfold with a symplectic automorphism of prime order greater than two, and  studies the geometry of such a cubic. This is a natural continuation of the mentioned series of works, and pursues the goal to have a completed geometrical and lattice theoretical picture. 
\begin{theorem}[{\autoref{algebraic_lattices}}]\label{algebraic_lattices}
    Let \(X\) be a general cubic fourfold among the ones admitting a symplectic automorphism \(\phi\) of prime order \(p \geq 3\). Then the Gram matrix of the lattice \(A(X)\) is described in \autoref{appendix_gram_matrices} and the lattice \(T(X)\) appears in \autoref{tab:order_p_sympl_cubic}. 
\end{theorem}
Moreover, starting from the structure of the algebraic lattice \(A(X)\) of a cubic fourfold with a symplectic automorphism of prime order at least three, we prove that it is generated by the square of the hyperplane class and by classes of cubic scrolls \autoref{cubics_with_scrolls} or planes \autoref{cubics_with_planes}. 
We obtain the following result.
\begin{corollary}[{\autoref{generated_by_scrolls}, \autoref{eckardt_points_in_phi_3^3}, \autoref{eckardt_points_in_phi_3^4}}]
    Let $X$ be a general cubic fourfold among the ones admitting a symplectic automorphism $\phi$ of prime order $p\geq 3$. Then the algebraic lattice $A(X)$ is generated by the square of the hyperplane class, along with either classes of cubic scrolls or planes.
\end{corollary}
The same phenomenon is already known to happen for any automorphism of order \(p=2\) by \cite{marquand2023cubic} (with exception of one case where \(\mathbb{Q}\)-coefficients are needed) and for non-symplectic automorphisms of order \(p\geq 3\) by \cite{billigrossi2024non}. 

One of the most challenging and intriguing problems about cubic fourfolds is to determine whether they are rational or not. It is conjectured that a very general cubic fourfold is not rational, but no examples of not rational cubic fourfolds are known. Moreover in \cite{hassett2000special,kuz2016derived} it is conjectured that a cubic fourfold is rational if and only if it has an associated K3 surface (i.e. its transcendental lattice is isometric to the transcendental lattice of a K3 surface). It was proved in \cite{Ouchi:Automorphisms.cubic.k3.category} (see also \autoref{Ouchi}) that a cubic fourfold with a symplectic automorphism of prime order \(p\geq 3\) has an associated K3 surface. As a consequence of our previous theorem, we prove the following result.
\begin{corollary}[{\autoref{rationality_symplectic}}]
     Let \(X\) be a cubic fourfold admitting a symplectic automorphism of prime order \(p\geq 3\), then \(X\) is rational.
\end{corollary}

Due to the presence of a group action, one can investigate an equivariant version of rationality and these conjectures, i.e. the $G$-rationality for $G\subset\Aut(X)$. 
More precisely, if \(G\subseteq\Aut(X)\) is a group of automorphisms of a cubic fourfold \(X\), then \(X\) is said to be \(G\)\textit{-rational} (or \(G\)-\textit{linearizable}) if \(G \subseteq PGL(5)\) is linear, and \(X\) admits a \(G\)-equivariant birational map to  \(\mathbb{P}^{4}\). Moreover in \cite[Theorem 1]{Bohning.Bothmer.Tschinkel:Equivariant.birational.geometry} the authors provide an example of a rational cubic fourfold that is not \(G\)-rational, where \(G\) is its group of symplectic automorphisms. 
We provide two families of symmetric cubics that are rational, but not $G$-rational:

\begin{theorem}[{\autoref{G_rationality}}]
    Let \(X\) be a cubic fourfold with a symplectic automorphism of order three of type \(\phi_3^3\) or \(\phi_3^4\) (see notation in \autoref{class_cubiche}) and let \(G=\Aut(X)\). Then \(X\) is rational, but not \(G\)-rational.
\end{theorem}
In a different vein, one can construct many examples of irreducible holomorphic symplectic (IHS) manifolds from a cubic fourfold $X$, and study birational transformations that are induced from the cubic.
 In particular, the families parametrizing \(1\)-cycles of degree zero and degree one on hyperplane sections of a cubic fourfold admit a compactification (see \cite{LSV,Voisin:compactification.twisted.2016hyper,Sacca2020birational}) which are IHS manifolds deformation equivalent to the O'Grady ten-dimensional sporadic example \cite{OGrady:desingularized.moduli.spaces}, i.e. it is an IHS manifold of \textit{OG10 type}. If \(X\) is a cubic fourfold, we refer to these constructions as \(J(X)\) a \textit{Laza-Saccà-Voisin (LSV) manifold} (the case of cycles of degree zero) and \(J^t(X)\) a \textit{twisted Laza-Saccà-Voisin (twisted LSV) manifold} (the case of cycles of degree one). 

An automorphism of a cubic fourfold induces a birational transformation of the corresponding OG10 type manifold. Symplectic birational transformations of IHS manifolds of OG10 type are classified in \cite{marquand2023classificationsymplecticbirationalinvolutions,marquand2024finitegroupssymplecticbirational} lattice-theoretically, and the authors classify those groups that are induced from a group of automorphisms of a cubic fourfold (see \cite[Theorem 6.1]{marquand2024finitegroupssymplecticbirational}). Further, in \cite[\S 6]{Felisetti.Giovenzana.Grossi} the authors give characterization of symplectic birational transformations the ones that are induced from a K3 surface. As an application of our results, we determine the Neron-Severi and transcendental lattices for a (twisted) LSV manifold associated to a cubic fourfold with a symplectic automorphism of prime order, strengthening the result of \cite[Theorem 6.1]{marquand2024finitegroupssymplecticbirational} in the prime order case.

\begin{theorem}[{\autoref{induced_action_LSV}}]
    Let \(X\) be a general cubic fourfold among the ones admitting a symplectic automorphism \(\phi\in\Aut(X)\) of prime order, then the Néron-Severi lattice and the transcendental lattice of the manifolds \(J(X)\) and \(J^t(X)\) are in \autoref{tab:order_p_sympl_LSV}.
\end{theorem}
\begin{remark}
    According to the theoretical obstruction proved in \cite[Theorem 1.1]{Giovenzana.Grossi.Onorati.Veniani} any non-trivial symplectic birational automorphism does not extend to a regular automorphism on an IHS manifold of OG10 type. By \cite[Proposition 3.11]{Sacca2020birational} the Lagrangian fibration of a LSV manifold will have a reducible fiber. We prove that any cubic fourfold with a symplectic automorphism contains planes or cubic scrolls, impling the existence of a reducible fiber by \cite{Brosnan} and \cite{marquand2024defectcubicthreefold}. 
\end{remark}

\subsection*{Outline}
In \autoref{preliminaries} we collect preliminaries about lattices and cubic fourfolds. In  \autoref{algebraic_lattice_section} we compute the lattices \(A(X)\) and \(T(X)\) for a general cubic fourfold \(X\) with a prime order symplectic automorphism. In  \autoref{cubics_with_scrolls} and  \autoref{cubics_with_planes} we study cubic fourfolds with a prime order symplectic automorphisms for which the lattice \(A(X)\) is generated by the square of a hyperplane class along with cubic scrolls or planes, respectively.
In  \autoref{G_rationality_section} we prove that if \(X\) is a cubic fourfold with an automorphism belonging to one of two particular families, then \(X\) is rational but it is not \(\Aut(X)\)-equivariantly rational. In  \autoref{LSV_birationalities} we determine the cohomological action of prime order symplectic automorphisms on manifolds of OG10 type that are induced by automorphisms of cubic fourfolds via the LSV constructions. In \autoref{appendix_gram_matrices} we display the matrices of the lattices \(A(X)\).
\subsection*{Acknowledgments}
We would like to thank Stevell Muller for discussions about lattice computations and for pointing out useful functions of \cite{OSCAR}. We also would like to thank Yuri Tschinkel for suggesting that the general cubic fourfold with an Eckardt involution is not linearizable, and Howard Nuer for valuable discussions on arrangements of planes. We are grateful to the referee for useful comments that improved this article, in particular for suggesting parts of the general \autoref{lem:Z3_discrim}, \autoref{lem:d_cong_2_mod_6}, and their proofs, and also many properties of the lattices in \autoref{appendix_gram_matrices}.

 \section{Preliminaries}\label{preliminaries}

\subsection{Lattices}A \textit{lattice} \(L\) is a free \(\mathbb{Z}\)-module of finite rank, with an integral symmetric bilinear form \( L\times L\to\mathbb{Z}\) which is non-degenerate. We use the notation \(x\cdot y\in \mathbb{Z}\) and \(x^2= x\cdot x\) for \(x,y\in L\). 
A lattice \(L\) is called \textit{even} if \(x^2\in 2\mathbb{Z}\) for any \(x\in L\), it is called \textit{odd} otherwise.
The \textit{signature} of \(L\) is the signature of its real extension. 
The divisibility of \(x\in L\) is the positive generator of the ideal \(x\cdot L\subseteq \mathbb{Z}\). 
Let \(L^\vee=\hom(L,\mathbb{Z})\), then the finite group \(D(L)=L^\vee/L\) is called the \textit{discriminant group}. We call \textit{discriminant} the quantity \(d(L):=|D(L)|\), it coincides with absolute value of the determinant of a Gram matrix of \(L\).
The \textit{length} \(l(D(L))\) is the minimum number of generators of \(D(L)\), if \(p\) is a prime number then the \(p\)\textit{-length} \(l_p(D(L))\) is the minimum number of generators of \(D(L\otimes\mathbb{Z}_{(p)})\). 
The bilinear form of \(L\) descends to a well-defined bilinear form \(D(L)\times D(L)\to \mathbb{Q}/\mathbb{Z}\). A lattice is called \textit{unimodular} if the discriminant group is trivial. A lattice \(L\) is called \(p\)\textit{-elementary} if \(D(L)\cong (\mathbb{Z}/p\mathbb{Z})^{\oplus a}\) for some integer \(a\geq 0\), in this case one has \(l(D(L))=l_p(D(L))\). The \textit{genus} of a lattice is the data of its parity, its signature, and its discriminant form. Isometric lattices must have the same genus, but lattices with the same genus can be not isometric.

An \textit{embedding} of lattices \(L\hookrightarrow M\) is a injective linear map that preserves the quadratic forms of the lattices \(L\) and \(M\). An embedding of lattices \(L\hookrightarrow M\) is \textit{primitive} if \(M/L\) is a free abelian group and in this case we denote by \(L^{\perp}\) the \textit{orthogonal complement} of \(L\) in \(M\).
If the embedding has finite index, we say that \(M\) is an \textit{overlattice} of \(L\). 
There is a correspondence between finite index overlattices of \(L\) and isotropic subgroups of \(D(L)\), by \cite[1.4.1]{Nikulin:int.sym.bilinear.forms}. 

A primitive embedding \(L\hookrightarrow M\) is determined by a group \(H \subseteq D(M)\) called the \textit{embedding subgroup}, an isometry \(\gamma\colon H\to H'\subseteq D(L)\) called the \textit{embedding isometry}. Similarly, by \cite[Proposition 1.5.1]{Nikulin:int.sym.bilinear.forms} a primitive embedding \(L\hookrightarrow M\) with \(L^\perp =T\) can be determined by a subgroup \(K\subseteq D(L)\) called the \textit{gluing subgroup} and an isometry \(\gamma\colon K\to K'\subseteq D(T(-1))\) called the \textit{gluing isometry}. We refer to \cite[Proposition 1.5.1, Proposition 1.15.1]{Nikulin:int.sym.bilinear.forms} for more details on embedding of lattices. 

The group of isometries of a lattice \(L\) is denoted by \(O(L)\). If \(\phi\in O(L)\) is an isometry of \(L\) then we let \(L^\phi=\{x\in L\mid \phi(x)=x\}\) be the \textit{invariant lattice} and \(L_\phi=(L^\phi)^\perp\subset L\) the \textit{coinvariant lattice}. They are primitive sublattices of \(L\), and if the isometry \(\phi\) is of finite order then \(L^\phi\oplus L_\phi\hookrightarrow L\) is a finite index embedding.  

We denote by \([k]\) the rank one lattice generated by an element of square \(k\in \mathbb{Z}\). We denote by \(\bU\) the even unimodular lattice of rank two, which is of signature \((1,1)\). We also denote by \(\bA_n,\bE_n,\bD_n\) the rank \(n\) positive definite lattices associated to the Dynkin diagrams ADE. For an odd integer \(n\geq 3\), we consider the indefinite lattice 
\[\bH_n=\begin{pmatrix}
    -2 & 1\\
    1 & (n-1)/2
\end{pmatrix},\]
which is even whenever \(n\equiv 3(4)\). Finally, if \(L\) is a lattice and \(k\) an integer, we denote by \(L(k)\) the lattice whose bilinear form is obtained from the one of \(L\) multiplied by \(k\).

 \subsection{Cubic fourfolds}
Let \(X\subset \mathbb{P}^5\) be a smooth cubic fourfold. The cohomology group \(H^4(X,\mathbb{Z})\) with the natural intersection pairing is the unique unimodular odd lattice \([1]^{\oplus 21}\oplus [-1]^{\oplus 2}\) of signature \((21,2)\). We denote by \(\eta_X\in H^4(X,\mathbb{Z})\) the square of the hyperplane class, and consider the primitive cohomology group \(H^4_{prim}(X,\mathbb{Z})=\langle \eta_X\rangle^\perp\). Notice that \(\langle \eta_X\rangle=[3],\) and the lattice \(H^4_{prim}(X,\mathbb{Z})\) is even. The primitive cohomology carries a polarized Hodge structure with Hodge numbers \((0,1,20,1,0)\), and we have an isometry of lattices \[H^4_{prim}(X,\mathbb{Z})\cong \bU^{\oplus2}\oplus \bE_8^{\oplus2}\oplus \bA_2.\] The \textit{algebraic lattice} is the lattice \(A(X)=H^4(X,\mathbb{Z})\cap H^{2,2}(X)\) and the \textit{transcendental lattice} is its orthogonal complement \(T(X)=A(X)^{\perp}\subseteq H^4(X,\mathbb{Z})\). The \textit{primitive algebraic lattice} is given by \(A_{prim}(X)=A(X)\cap H^4_{prim}(X,\mathbb{Z})\subseteq H^4_{prim}(X,\mathbb{Z})\). 

\begin{remark}\label{rmk:roots}
    Note that for a smooth cubic fourfold $X$, there does not exist $v\in A_{prim}(X)$ with $v^2=2$ (a short root) or with $v^2=6$ and divisibility 3 in the lattice $H^4(X,\mathbb{Z})$ (a long root). This follows from the description of the image of the period map \cite[Theorem 1.1]{laza2010moduli}.
\end{remark}

Let $\phi\in \Aut(X),$ and consider the induced action $\phi^*\in O(H^4(X,\mathbb{Z})).$ We thus obtain a map
 \[\Aut(X)\to O(H^4(X,\mathbb{Z}))\]
which is injective (combine \cite[Prop. 2.12]{MR3673652}, \cite{Matstumura_Monsky}). Further, we have that \(\Aut(X)\cong O_{Hdg}(H^4(X,\mathbb{Z}),\eta_X)\), the group of Hodge isometries preserving the class \(\eta_X\) by \cite{Zheng}.
\begin{definition}
    An automorphism \(\phi\in\Aut(X)\) is called \textit{symplectic} if it acts trivially on \(H^{3,1}(X,\mathbb{Z})\), \textit{non-symplectic} otherwise.
\end{definition}
We denote by $\Aut_S(X)$ the subgroup of $\Aut(X)$ consisting of symplectic automorphisms.
 According to \cite[\S 4.1]{Laza.Zheng:aut} (see also \cite[§3]{Nikulin:finite.aut.groups.K3} for the similar case of $K3$ surfaces), we have \(H^4_{prim}(X,\mathbb{Z})_\phi\subseteq A_{prim}(X)\) (and hence \(T(X)\subseteq H^4_{prim}(X,\mathbb{Z})^\phi\)) if \(\phi\) is symplectic. We say that the cubic fourfold \(X\) endowed with the action of an automorphism \(\phi\in\Aut(X)\) is \textit{general} if the previous inclusion is an equality.

The group of automorphisms of a cubic fourfold is a finite group of linear automorphisms \cite{Matstumura_Monsky}. A complete classification of automorphisms of prime order is obtained in \cite{GonzalezAguilera.Liendo:aut.prime.ord.smooth.cubic.n-folds}. In the following we recall the classification of symplectic automorphisms of prime order of a cubic fourfold (see also \cite{fu2016classification}).

\begin{theorem}[see \cite{GonzalezAguilera.Liendo:aut.prime.ord.smooth.cubic.n-folds}, {\cite[Proposition 6.1]{yu2020moduli}}]\label{class_cubiche}
Let \(X = \{F=0\} \subset \mathbb{P}^5\) be a smooth cubic fourfold with a symplectic automorphism \(\phi\in\Aut(X)\) of prime order \(p\). After a linear change of coordinates that diagonalizes \(\phi\), we have \(\phi(x_0:\ldots :x_5)=(\xi^{\sigma_0} x_0:\ldots :\xi^{\sigma_5} x_5)\) and we denote by \((\sigma_0, \ldots, \sigma_5)\) such an action. If \(d\) denotes the dimension of the family \(F_p^i\) of cubic fourfolds endowed with the automorphism \(\phi_p^i\), then we have the following possibilities:
\vspace{3pt}
\begin{itemize}\small{
        \item \(\phi_2^2\): \(p=2\), \(\sigma=(0,0,0,0,1,1)\), \(d=12\),
        \[F=L_3(x_0,\dots,x_3)+x_4^2 L_1(x_0,\dots,x_3)+x_4x_5M_1(x_0,\dots,x_3)+x_5^2 N_1(x_0,\dots,x_3),\]

        \item \(\phi_3^3\): \(p=3\), \(\sigma=(0,0,0,0,1,2)\), \(d=8\),
        \[F=L_3(x_0,\dots,x_3)+x_4^3+x_5^3+x_4x_5M_1(x_0,\dots,x_3),\]

        \item \(\phi_3^4\): \(p=3\), \(\sigma=(0,0,0,1,1,1)\), \(d=2\),
        \[F=L_3(x_0,x_1,x_2)+M_3(x_3,x_4,x_5),\]

        \item \(\phi_3^6\): \(p=3\), \(\sigma=(0,0,1,1,2,2)\), \(d=8\),
        \[F=L_3(x_0,x_1)+M_3(x_2,x_3)+N_3(x_4,x_5)+\sum_{i=0,1;j=2,3;k=4,5}a_{i,j,k}x_ix_jx_k,\]

        \item \(\phi_5^1\): \(p=5\), \(\sigma=(0,0,1,2,3,4)\), \(d=4\),
        \[F=L_3(x_0,x_1)+x_2x_5L_1(x_0,x_1)+x_3x_4M_1(x_0,x_1)+x_2^2x_4+x_2x_3^2+x_3x_5^2+x_4^2x_5,\]

        \item \(\phi_7^1\): \(p=7\), \(\sigma=(1,2,3,4,5,6)\), \(d=2\)
\[F=x_0^2x_4+x_1^2x_2+x_0x_2^2+x_3^2x_5+x_3x_4^2+x_1x_5^2+ax_0x_1x_3+bx_2x_4x_5\]

\item \(\phi_{11}^1\): \(p=11\), \(\sigma=(0,1,3,4,5,9)\), \(d=0\),
        \[F=x_0^3+x_1^2x_5+x_2^2x_4+x_2x_3^2+x_1x_4^2+x_3x_5^2,\]}
    \end{itemize}
where \(L_i,M_i,N_i\) and \(O_i\) are homogeneous polynomials of degree \(i\), \(a_{i,j,k},a,b \in\mathbb{C}\).
\end{theorem}

For a smooth cubic fourfold \(X\subset \mathbb{P}^5\), the Fano variety of lines $F(X)$ is an irreducible holomorphic symplectic fourfold, which is of \(K3^{[2]}\)-type \cite{Beauville.Donagi:variete.droites.hypersurf.cub.dim.4}.
The second cohomology $H^2(F(X),\mathbb{Z})$ is equipped with the Beauville-Bogomolov-Fujiki form \(q\), which is an integral bilinear quadratic form.  
 We let $H^2_{prim}(F(X),\mathbb{Z}):=g^\perp\subset H^2(F(X),\mathbb{Z})$, where $g$ is the Pl\"ucker polarisation, satisfying $q(g)=-6$.
The Abel-Jacobi map \(\alpha\colon H^4(X,\mathbb{Z})\rightarrow H^2(F(X),\mathbb{Z})\) restricts to an isomorphism of Hodge structures \[H^4_{prim}(X,\mathbb{Z})\rightarrow H^2_{prim}(F(X),\mathbb{Z})\] satisfying $q(\alpha(x), \alpha(y))=-x\cdot y$ for $x,y\in H^4_{prim}(X,\mathbb{Z}).$ 

\begin{remark}
    Note that an automorphism $f$ of $X$ induces an automorphism of $F(X),$ which is symplectic if and only if $f\in \mathrm{Aut}_S(X)$, by the discussion above.
\end{remark}

\section{The algebraic lattice}\label{algebraic_lattice_section}

In this section we compute the algebraic lattice \(A(X)\) and the transcendental lattice \(T(X)\) for a cubic fourfold \(X\) which is general in a family \(F_p^i\) as in \autoref{class_cubiche}. The primitive algebraic lattice \(A_{prim}(X)\) has been classified in \cite{Laza.Zheng:aut}, we need to understand the relation of \(A(X)\) with the sublattice \(A_{prim}(X)\). 

Denote by \(\mathcal{C}_d\) the irreducible \textit{Hassett divisor} of cubic fourfolds \(X\) admitting a primitive rank \(2\) sublattice \(\eta_X\in K\subseteq A(X)\) of discriminant \(d\in\mathbb{Z}_{\geq 0}\) and recall that \(\mathcal{C}_d\) is non-empty if and only if \(d\equiv 0 (6)\) or \(d\equiv 2 (6)\), see \cite{hassett2000special}.

\begin{lemma}\label{lem:Z3_discrim}
    Let \(X\) be a cubic fourfold, and let \(\Gamma\subseteq H_{prim}^4(X,\mathbb{Z}) \) be a primitive sublattice. Then, there exists a proper overlattice \(\langle \eta_X\rangle\oplus \Gamma\subset \widetilde{\Gamma} \) obtained by gluing the discriminant form of \(\langle\eta_X\rangle\) if and only if \(D(\Gamma)\) admits a subgroup isomorphic to \(\mathbb{Z}/3\mathbb{Z}\) with discriminant form \(\frac{-1}{3}\). In this case, the overlattice \(\Gamma\oplus \langle \eta_X\rangle\subset \widetilde{\Gamma}\) has index \(3\) and \(d(\Gamma)=3 \cdot d(\widetilde{\Gamma})\).
\end{lemma}
\begin{proof}
    Assume there is such an overlattice \(\Gamma\oplus \langle \eta_X\rangle \subset \widetilde{\Gamma}\). 
    Then \(\widetilde{\Gamma}\) is generated by \(\eta_X\), the generators of \(\Gamma\), and a vector \[v=\frac{k \eta_X+\omega}{r}\] with \(\omega\in \Gamma\), \(r,k \in \mathbb{N}\) such that \(r\not=0,1\). 
    In this case, \(r\) is the index of \(\Gamma\oplus \langle \eta_X\rangle\subset \widetilde{\Gamma}\). 
    Note that \(v\cdot\eta_X=\frac{3k}{r}\) is an integer and, by fact that \(\Gamma\) is primitive in \(\widetilde{\Gamma}\), we see that \(r\) cannot divide \(k\). 
    In other words, \(\frac{k\eta_X}{r}\) must be a non-trivial element of \(D(\Gamma\oplus \langle\eta_X\rangle)\), hence \(k\not\equiv 0 (3)\) and \(r=3\). 
    As a consequence, we can choose \(v\) of the form \(v=\frac{\eta_X+\omega}{3}\) (substituting \(v\) with \(-2v+\eta_X\)
 when \(k\equiv 2 (3)\)) and, since for any \(x\in \Gamma\) we have \(v\cdot x=\frac{\omega\cdot x}{3}\in \mathbb{Z}\), we deduce that \[\gamma:= \frac{\omega}{3}\in D(\Gamma)\] is well defined of order \(3\). We compute \(v^2=\frac{3+\omega^2}{9}\in \mathbb{Z}\), which implies that \[\gamma^{2}=\frac{\omega^2}{9}=\frac{-1}{3} \mod \mathbb{Z}\] since \(\omega^2\equiv -3  (9)\). From the fact that taking a finite index \(r\) overlattice the discriminant is divided by \(r^2\), we have \[\frac{ d(\widetilde{\Gamma})}{9}=\frac{d(\langle\eta_X\rangle\oplus \Gamma)}{9}=\frac{3 d(\Gamma)}{9}\] from which we deduce that \(d(\Gamma)=3 d(\widetilde{\Gamma})\).

 Conversely, following the same proof, the existence of such an element \(\gamma\in D(\Gamma)\) gives a vector \(v=\frac{\eta_X}{3}+\gamma\) which determines the overlattice \(\widetilde{\Gamma}\) with the required properties.
 \end{proof}

\begin{remark}
    Note that the overlattice \(\widetilde{\Gamma}\) as in the previous lemma may or may not be contained in \(H^4(X,\mathbb{Z})\). This situation will be clarified by analysing the cases \(\phi_3^3\) and \(\phi_3^6\), respectively.
\end{remark}

\begin{lemma}\label{lem:d_cong_2_mod_6}
    Let \(X\) be a cubic fourfold, then \(A(X)\) is a proper overlattice of \(\langle \eta_X\rangle\oplus A_{prim}(X)\) if and only if \(X\in\mathcal{C}_d\) for some \(d\equiv 2 (6)\). In this case, the embedding of lattices \(\langle \eta_X\rangle\oplus A_{prim}(X) \subset A(X)\) has index \(3\) and \(3\cdot  d(A(X))= d(A_{prim}(X))\).
\end{lemma}

\begin{proof}
Clearly we assume \(\rank A(X)\geq 2\), since otherwise \(A(X)=\langle\eta_X\rangle\) and there is nothing to prove.
Let \(\Gamma=A_{prim}(X)\) then, by \autoref{lem:Z3_discrim} and its proof, there exists a proper overlattice \(\widetilde{\Gamma}\) of \(\langle \eta_X\rangle\oplus A_{prim}(X)\) obtained by gluing the discriminant group of \(\langle\eta_X\rangle\) if and only if there exists an element \(\omega \in A_{prim}(X)\) such that \(\omega^2\equiv -3 (9)\). In this case, the overlattice \(\widetilde{\Gamma}\) coincides with \(A(X)\) if and only if it is contained in \(H^4(X,\mathbb{Z})\).

Suppose first that \(A(X)=\langle \eta_X\rangle\oplus A_{prim}(X)\) and let \(\eta_X \in K\subseteq A(X) \) be a primitive sublattice of rank \(2\). Then \(K=\langle\eta_X,a \eta_X+v\rangle\) for \(0\not=v\in A_{prim}(X)\) and \(a\in\mathbb{Z}\). We can suppose \(K=\langle\eta_X,v\rangle\) after applying a linear transformation. The discriminant  of \(K\) is given by \(d=3k\) for \(k\in 2\mathbb{Z}\), and hence \(d\equiv 0 (6)\). 
This proves that if \(X\in \mathcal{C}_d\) for \(d\equiv 2 (6)\), then \(A(X)\) is a proper overlattice of \(\langle\eta_X\rangle\oplus A_{prim}(X)\). Thus we now assume that \(A(X)\) is a proper overlattice of \(\langle\eta_X\rangle\oplus A_{prim}(X)\), and we willl to prove that \(X\in \mathcal{C}_d\) for some \(d\equiv 2(6)\).

By the previous discussion the finite index embedding \(\langle\eta_X\rangle\oplus A_{prim}(X)\subset A(X)\) must be of index \(3\) and \(3\cdot  d(A(X))= d(A_{prim}(X))\). Let \( \omega\in A_{prim}(X)\) be the element determining the finite index embedding as in the proof of \autoref{lem:Z3_discrim}, then write \(\omega= s H\) for a primitive vector \(H\in A_{prim}(X)\) and \(s\in \mathbb{Z}\). We have that \(H^2=2a\) for \(a\in \mathbb{Z}\) since \(A_{prim}(X)\) is an even lattice. 

On the one hand, we have that \(\omega^2=s^2 2a\), and on the other hand we know that \(\omega^2\equiv -3(9)\) so that \(\omega^2=-3+9(2h+1)=6+18h\) for \(h\in \mathbb{Z}\). The squares modulo \(9\) are \(0,1,4,7\), and hence \(s^2 2a=6+18h\) if and only if \(a\equiv 3 (9)\) or \( a\equiv 12(36)\). In these cases, we see that there exists an overlattice \(K\subseteq A(X)\) of \(\langle\eta_X,H\rangle\) whose discriminant is \[d=\frac{d(\langle H\rangle)}{3}=\frac{2a}{3},\] by applying again \autoref{lem:Z3_discrim} for \(\Gamma=\langle H\rangle\) together with the fact that in this case \(\widetilde{\Gamma}\subseteq H^4(X,\mathbb{Z})\). Now, from \(a=3+9k\), \(k\in \mathbb{Z}\) we get \[d=\frac{2(3+9k)}{3}=2+6k,\] and from \(a=12+36k\), \(k\in \mathbb{Z}\) we get \[d=\frac{2(12+36k)}{3}=2+6(4k+1),\] so that in any case \(d\equiv 2(6)\).

\end{proof}

As a consequence, we see that when $X$ admits an automorphism of prime order different from three, the lattice \(A_{prim}(X)\) determines the lattice \(A(X)\).

\begin{proposition}\label{p_not_3}
    Let \(X\) be a cubic fourfold with a symplectic automorphism \(\phi_p^i\) of prime order \(p\not=3\), which is general in the family \(F_p^i\). Then, we have \(A(X)=\langle \eta_X\rangle\oplus A_{prim}(X)\) and \(X\not\in\bigcup_{d\equiv 2(6)}\mathcal{C}_d\).
\end{proposition}
\begin{proof}
    We know that if \(X\) is general and \(\phi_p^i\in\Aut(X)\) is symplectic then \(A_{prim}(X)=H^4(X,\mathbb{Z})_{\phi_p^i}\). Moreover, since \(\phi_p^i\) has prime order \(p\) then \(A_{prim}(X)\) is \(p\)-elementary by \cite[Lemma 1.8]{Mongardi.Tari.Wandel:kummer.fourfolds}. We observe that either \(A(X)=\langle\eta_X\rangle\oplus A_{prim}(X)\) or \(A(X)\) is a proper overlattice of \(\langle\eta_X\rangle\oplus A_{prim}(X)\) obtained by gluing the iscriminant form of \(\langle\eta_X\rangle\). As a consequence, if \(p\not = 3\) then \(D(A_{prim}(X))\) has no subgroup isomorphic to \(\mathbb{Z}/3\mathbb{Z}\) so that by \autoref{lem:Z3_discrim} we have \(A(X)=\langle \eta_X\rangle\oplus A_{prim}(X)\) and by \autoref{lem:d_cong_2_mod_6} it follows that \(X\not\in\bigcup_{d\equiv 2(6)}\mathcal{C}_d\).
\end{proof}

\begin{theorem}\label{algebraic_lattices}
    Let \(X\) be a cubic fourfold with a symplectic automorphism \(\phi_p^i\) of prime order \(p\), which is general in the family \(F_p^i\). Then the lattice \(A(X)\) has Gram matrix as in \autoref{appendix_gram_matrices} and the lattice \(T(X)\) is displayed in \autoref{tab:order_p_sympl_cubic}. 
\end{theorem}
\begin{proof}
The automorphism \(\phi_p^i\) induces an isometry on \(H_{prim}^4(X,\mathbb{Z})\) and \((\phi_p^i,H_{prim}^4(X,\mathbb{Z})_{\phi_p^i})\) is a Leech pair (cf. \cite[Definition 3.3]{Laza.Zheng:aut}). Note that for a general cubic fourfold with the action of \(\phi_p^i\) it holds that \(A_{prim}(X)=H_{prim}^4(X,\mathbb{Z})_{\phi_p^i}\) and the lattice is in the classification \cite[Theorem 1.2]{Laza.Zheng:aut}. We recall that the Leech lattice has no vectors of square \(2\), ensuring that \(A_{prim}(X)\) contains no short roots accordingly to \autoref{rmk:roots}, while the condition of not containing long roots has its relevance in the computation of \(A(X)\). The lattice \(A_{prim}(X)\) is \(p\)-elementary and it corresponds to the cases No. 20, 52, and 120 of \cite[Table 1]{Hoehn.Mason:290.fixed.Leech} for the primes \(5,7\) and \(11\). 
By \autoref{p_not_3}, we deduce that \(A(X)=\langle\eta_X \rangle\oplus A_{prim}(X)\) whenever \(p\not=3\). 
It remains to discuss automorphisms of order \(p=3\). Here we use a case-by-case analysis. 
\begin{enumerate}
    \item Consider the Leech pair \((G,S)\) no. \(4\) in \cite[Table 1]{Hoehn.Mason:290.fixed.Leech} with \(G\) of order three. 
    The lattice \(S\) is the coinvariant lattice with respect to a Leech pair \((G,S)\) where \(G\) is a group of order three and \(S\) has rank \(12\). The lattice \(S\) is uniquely determined up to isometry and it turns out to be the same for the action of \(\phi_3^3\) and \(\phi_3^6\) on general cubic fourfolds in the families \(F_3^3\) and \(F_3^6\). We have \(S \cong A_{prim}(X_3^3)=A_{prim}(X_3^6)\) where \(X_3^3\in F_3^3\) and \(X_3^6\in F_3^6\) are general cubic fourfold in the families with automorphisms \(\phi_3^3\) and \(\phi_3^6\), respectively. 
    There are two possibilities for the algebraic lattice: it is either isometric to \([3]\oplus S\) or to an index \(3\) overlattice of that. 
    We observe that the cubic \(X_3^3\) contains planes (see \autoref{eckardt_points_in_phi_3^3} for more detail), and hence \(X_3^3\in\mathcal{C}_8\). 
    By \autoref{lem:Z3_discrim} and \autoref{lem:d_cong_2_mod_6}, it follows that the embedding \(\langle\eta_{X_3^3}\rangle\oplus A_{prim}(X_3^3) \subset A(X_3^3) \) is of index \(3\). To compute the lattice \(A(X_3^3)\), we compute the possible extensions of index \(3\) of \(\langle\eta_{X_3^3}\rangle\oplus A_{prim}(X_3^3)\) and exclude the cases where there are vectors of square \(1\), which correspond to having long roots in \(A_{prim}(X)\) by \cite[Lemma 4.4]{billigrossi2024non}, according to \autoref{rmk:roots}. This can be easily done for example using \cite{OSCAR}, however \(A(X_3^3)\) is the only index three extension of \(\langle\eta_{X_3^3}\rangle\oplus A_{prim}(X_3^3)\) with no vectors of square \(1\) and hence it suffices to exhibit such a lattice. We deduce that \(X_3^6\) corresponds to the other case, for which we have \(A(X_3^6)=\langle\eta_{X_3^6}\rangle\oplus A_{prim}(X_3^6)\).
    \item Consider the Leech pair \((G,S)\) no. 35 in \cite[Table 1]{Hoehn.Mason:290.fixed.Leech} with \(G\) of order three. 
    Let \(X_3^4\) be a general cubic fourfold in \(F_3^4\); we have that \(A_{prim}(X_3^4)\cong S\) has rank \(18\) and it is \(3\)-elementary with \(l(D(A_{prim}(X_3^4)))=5\). 
    In this case the transcendental lattice \(T(X_3^4)\) has rank \(4\) and, as a consequence, its orthogonal complement \(A(X_3^4)\) in the unimodular lattice \(H^4(X_3^4,\mathbb{Z})\) must have length \(l(D(A(X)))\leq 4\).
    It follows that the discriminant form of \(\langle\eta_X\rangle\) must be glued and \(\langle\eta_X\rangle\oplus  A_{prim}(X_3^4)\subset  A(X_3^4,\mathbb{Z})\) has index \(3\), as in \autoref{lem:d_cong_2_mod_6}. We compute the lattice \(A(X_3^4)\) by the same technique as before. Alternatively we also know that the cubic \(X_3^4\) contains planes, see \cite{koike2022cubic}, and we can argue as before. 
\end{enumerate}
We computed the lattice \(A(X)\) for any general cubic fourfold \(X\in F_p^i\), and the transcendental lattice \(T(X)\) is uniquely determined by the rank and the discriminant form of \(A(X)\). In fact we see that, in all cases except \(X\in F_3^4, F_7^1,F_{11}^1\), the lattice \(T(X)\) is indefinite with \(\rank T(X)\geq 3\) and \(\rank T(X)\geq l_q(T(X))+2\) for any prime \(q\), so that by \cite[Theorem 1.13.2]{Nikulin:int.sym.bilinear.forms} the lattice \(T(X)\) is uniquely determined. For the remaining cases, we need to argue separately. From \(l(A(X_3^4))=4\) we know that \(T(X_3^4)\) is 3-elementary with \(l(T(X_3^4))=4\), moreover we know that the discriminant form \(\delta_p\) of \(D(T(X_{p}^1))\) is anti-isometric to the one of \( D(A(X_{p}^1))\oplus D([3])\) for \(p=7,11\). We now use again \cite{OSCAR} to check the following: there is only one genus of even 3-elementary lattices of signature \((2,2)\) and length \(4\) with a unique isometry representative, there is a unique even lattice of signature \((2,2)\) with disciminant form \(\delta_7\), and there is a unique even lattice of signature \((0,2)\) with discriminant form \(\delta_{11}\). 
\end{proof}

\begin{center}
{\begin{longtable}{lllllll}
\caption{Description of the pairs \((A(X),T(X))\) for a general cubic fourfold \(X\in F_p^i\) with a symplectic automorphism \(\phi_p^i\) of prime order \(p\). The lattices \(A(X)\) are in \autoref{appendix_gram_matrices}.}
\label{tab:order_p_sympl_cubic} \\
		
		\toprule
	 No. & \(\rank(A(X))\) &  \(T(X)\)  & \(\sign(T(X))\) & \(l_p(D(A(X)))\) & \(p\)  \\

		\midrule
		\endfirsthead
		
		\multicolumn{7}{c}%
		{\tablename\ \thetable{}, follows from previous page} \\
	 No. & \(\rank(A(X))\) &  \(T(X)\)  & \(\sign(T(X))\) & \(l_p(D(A(X)))\) & \(p\)  \\
		
		\midrule
		\endhead
	
		\multicolumn{7}{c}{Continues on next page} \\
		\endfoot
		
		\bottomrule
		\endlastfoot
\(\phi_3^3\) & \(13\) & \( \bU(3)^{\oplus 2}\oplus\bE_6\)
& \((8,2)\) & \(5\)  & \(3\) \\

\(\phi_3^6\) & \(13\) & \( \bU(3)^{\oplus 2}\oplus\bA_2^{\oplus 3} \)
& \((8,2)\) & \(7\)  & \(3\) \\

\(\phi^4_3\) & \(19\) & \(\bU(3)^{\oplus 2} \)
& \((2,2)\) & \(4\)  & \(3\) \\
\hline
\(\phi^1_5\) & \(17\) & \(\bU(5)^{\oplus 2}\oplus \bA_2\) 
& \((4,2)\) & \(4\)  & \(5\) \\
\hline
\(\phi^1_7\) & \(19\) & \(\bU(7)\oplus \bH_{21}\) 
& \((2,2)\) & \(3\)  & \(7\) \\

\hline
\(\phi^1_{11}\) & \(21\) & \(\bA_2(-11)\)  
& \((0,2)\) & \(2\)  & \(11\) \\

	\midrule
    	\end{longtable}}
\end{center}

\begin{remark}\label{Ouchi}
In particular, one can see that in all cases there is a cohomologically associated K3 surface. It was already proved in \cite{Ouchi:Automorphisms.cubic.k3.category} that a cubic fourfold admitting a symplectic automorphism of order different from two has an associated K3 surface in the derived sense. This is equivalent to have an associated K3 surface in the cohomological sense by \cite[Theorem 1.1] {Addington.Thomas} and \cite[Corollary 1.7]{BLM21}.
\end{remark}
\begin{corollary}\label{rationality_symplectic}
    Let \(X\) be a cubic fourfold admitting a symplectic automorphism of prime order \(p\geq 3\), then \(X\) is rational.
\end{corollary}
\begin{proof}
    We use the description of the Gram matrix \(A(X)\) in \autoref{appendix_gram_matrices} for a general cubic fourfold \(X\) in one of the families \(F_p^i\) to show that \(X\in\mathcal{C}_{14}\) or \(X\in \mathcal{C}_{42}\).  By \cite{bolognesi2019some} and \cite{RS_trisecant} we know that cubic fourfolds in these Hassett divisors are rational. If \(X\in F_3^3\) or \(X\in F_3^4\), then \(X\) contains disjoint planes by \autoref{eckardt_points_in_phi_3^3} and \autoref{eckardt_points_in_phi_3^4}, hence \(X\) is rational and in particular \(X\in \mathcal{C}_{14}\). In all the other cases we have \(A(X)=\langle \eta_X \rangle\oplus A_{prim}(X)\) and \(X\not \in \mathcal{C}_{14}\), but it is easy to find a primitive sublattice of \(A_{prim}(X)\) with Gram matrix 
    \[\begin{pmatrix}
    4 & 1 & 0 \\
    1 &4 & 0\\
    0 & 0 & 4
\end{pmatrix},\] the sum of the generators is a vector \(v\in A_{prim}(X)\) with \(v^2=14\) and the lattice \(K=\langle \eta_X,v\rangle \subseteq A(X)\) gives \(X\in \mathcal{C}_{42}\). 
\end{proof}
\section{Cubic fourfolds containing cubic scrolls}\label{cubics_with_scrolls}

All the cubic fourfolds with a symplectic automorphism of prime order belong to the divisor \(\mathcal{C}_{12}\), which is the closure of the locus of cubic fourfolds containing a cubic scroll. By \autoref{lem:d_cong_2_mod_6}, \autoref{p_not_3} and \autoref{algebraic_lattices}, the general cubic with automorphism \(\phi_3^6,\phi_5^1,\phi_5^1\) or \(\phi_{11}^1\) does not belong to \(\mathcal{C}_8\) and hence it does not contain planes. In the following we prove that these cubics contain cubic scrolls, and their classes generate the algebraic lattice.

\begin{proposition} \label{generated_by_scrolls}
    Let \(X\) be a general cubic fourfold in one of the families \(F_3^6,F_5^1,F_7^1\) or \(F_{11}^1\). Then \(X\) contains \(k\) families of cubic scrolls \(\{T_i,T^\vee_i\}^{k}_{i=1}\) such that \([T_i]+[T^\vee_i]=2\eta_X\) where \(k=378, 1320, 2709\) or \(6270\) respectively. Moreover, the algebraic lattice \(A(X)\) is generated by the classes \([T_i]\) for \(i=1,\dots,k\).
\end{proposition}
\begin{proof}
    We argue as in the proofs of \cite[Proposition 4.10, Proposition 4.11]{billigrossi2024non}: using the proof of 
\cite[Lemma 2.11]{brooke2025lines} classes of cubic scrolls in \(X\) correspond to classes of extremal rational curves in the Fano variety of lines \(F(X)\) with the right numerical properties, i.e. square \(-10\) and divisibility \(2\) in \(H^{1,1}(F(X),\mathbb{Z})_{prim}\) which correspond to vectors of square \(4\) in \(A_{prim}(X)\). The families of cubic scrolls come in pairs and, using the matrix description of the lattice \(A_{prim}(X)=\langle\eta_X\rangle^{\perp_{A(X)}}\) in \autoref{algebraic_lattices}, it is easy to check (using a computer) that there are \(2k\) such vectors which generate the entire lattice. 
\end{proof}

\begin{remark}
    The number of families of cubic scrolls was also computed implicitly in \cite[Table 5]{marquand2024finitegroupssymplecticbirational}.
\end{remark}

The group of symplectic automorphisms of the cubic we consider is known by lattice theoretical arguments, see \cite{Laza.Zheng:aut}. For a general cubic fourfold \(X\) in the family \(F_3^6\) we have that \(\Aut_S(X)=\langle\phi_3^6\rangle\), but in the other cases the group of symplectic automorphisms is bigger and we give explicit generators for the group.

\begin{proposition}\label{prop: aut grp}
There is the following description of automorphism groups:
\begin{itemize}
    \item  Let \(X\) be a general cubic fourfold in the family \(F_5^1\), then \(\Aut_S(X)\cong D_5\) the dihedral group (of \(10\) elements) and it is generated by \(\phi_5^1\) together with the symplectic involution
    \[ \tau=\begin{pmatrix}
    1 & 0 & 0 & 0 & 0 & 0 \\ 
    0 & 1 & 0 & 0 & 0 & 0\\
    0 & 0 & 0 & 0 & 0 & 1\\
    0 & 0 & 0 & 0 & 1 & 0\\
    0 & 0 & 0 & 1 & 0 & 0\\
    0 & 0 & 1 & 0 & 0 & 0
\end{pmatrix}.
\] In particular, there is an inclusion of families \(F_5^1\subset F_2^2\).
\item Let \(X\) be a general cubic in the family \(F_7^1\), then \(\Aut_S(X)= \langle \phi_7^1\rangle \rtimes \langle\tau\rangle \) where
\[\tau=\begin{pmatrix}
    0 & 1 & 0 & 0 & 0 & 0 \\ 
    0 & 0 & 0 & 1 & 0 & 0\\
    0 & 0 & 0 & 0 & 0 & 1\\
    1 & 0 & 0 & 0 & 0 & 0\\
    0 & 0 & 1 & 0 & 0 & 0\\
    0 & 0 & 0 & 0 & 1 & 0
\end{pmatrix}\] is an order three symplectic automorphism. In particular, there is an inclusion of families \(F_7^1\subset F_3^6\).
\item Let \(X\) be a general cubic fourfold in the family \(F_{11}^1\), then \(\Aut_S(X)\cong L_2(11)\) it is generated by \(\phi_{11}^1\) and the symplectic automorphism of order five 
\[ \tau=\begin{pmatrix}
    1 & 0 & 0 & 0 & 0 & 0 \\ 
    0 & 0 & 1 & 0 & 0 & 0\\
    0 & 0 & 0 & 0 & 0 & 1\\
    0 & 1 & 0 & 0 & 0 & 0\\
    0 & 0 & 0 & 1 & 0 & 0\\
    0 & 0 & 0 & 0 & 1 & 0
\end{pmatrix}.\] Moreover, we have that \( X\in F_2^2\cap F_3^6 \cap F_5^1\).
\end{itemize}
\end{proposition}
\begin{proof}
It is easy to check that in all cases \(\tau\) is symplectic, it has the wanted order, and that it satisfies the group relations. We consider the case of the automorphism \(\phi^1_{11}\), where the statement is not as direct as the other cases.
    The matrix description is deduced by the description of its unique irreducible \(5\)-dimensional representation (see for example \cite{ATLAS}). The symplectic automorphism \(\tau\) has order \(5\), hence \(X\in F_5^1\). The group \(L_2(11)\) has order \(660\), hence it admits symplectic automorphisms of order \(2\) and \(3\). It follows that \(X\in F_2^2\) and that \(X\in F_3^6\) since cubic fourfolds in \(F_3^3\) and \(F_3^4\) contain planes, but \(X\not\in\mathcal{C}_8\) by \autoref{p_not_3}. 
\end{proof}

 \begin{remark}
     Recently, the possibilities for the full group of automorphisms for a cubic fourfold in the family $F_7^1$ have been classified \cite{he2025cubicfourfoldsorder7automorphism}.
 \end{remark}

\begin{remark}
    Note that if \(F(X)\) admits a symplectic automorphism of order \(p=5,7,11\) that is induced by an automorphism of cubic fourfold \(X\), then the entire group \(G=D_{10}, C_7:C_3,L_2(11)\) is a symplectic group of automorphisms of \(F(X)\) and is completely induced by automorphisms of \(X\). Induced automorphisms on irreducible holomorphic symplectic varieties of $K3^{[2]}$ type have been well studied, for more examples see \cite{Mong13,Hoehn.Mason:finite.groups.sp.aut.K3^2,fu2016classification}. 
\end{remark}

The cubic \(X\) with automorphism \(\phi_{11}^1\) is the triple cover of \(\mathbb{P}^4\) ramified on the Klein cubic threefold, the group of symplectic automorphisms is known to be isomorphic to \(L_2(11)=\PSL(2,\mathbb{F}_{11})\), see \cite{Adler,Laza.Zheng:aut}. The group \(\Aut(X)/\Aut_S(X)\) has order three generated by the covering automorphism, in particular \(X\in F_3^1\).

 We denote by \(Y_A\) the double EPW-sextic associated to the Klein Lagrangian \(A\) as in \cite{debarre2022gushel}, with symplectic group of polarized automorphisms given by \(L_2(11)\). We also observe that for a cubic fourfold \(X\) with automorphism \(\phi_{11}^1\), the group of symplectic polarized automorphisms of the Fano variety of lines \(F(X)\) is also given by \(L_2(11)\). 
 \begin{proposition}
     Let \(X\) be a cubic fourfold with automorphism \(\phi_{11}^1\). Then the Fano variety of lines \(F(X)\) is not birational to the double EPW-sextic \(Y_A\).
 \end{proposition}
 \begin{proof}
     According to \cite{debarre2022gushel}, the transcendental lattice of \(Y_A\) is given by \[T(Y_A)\cong[22]\oplus[22]\cong\begin{pmatrix}
    22 & 0 \\
    0 &22
\end{pmatrix},\]

while the transcendental lattice of \(F(X)\) is 
\[T(F(X))\cong\bA_2(11)\cong \begin{pmatrix}
    22 & 11 \\
    11 &22
\end{pmatrix}.\]
 \end{proof}

\section{Cubic fourfolds containing planes}\label{cubics_with_planes}

In this section we study cubic fourfolds with automorphisms \(\phi_3^3\) and \(\phi_3^4\). We prove that the general such cubic contains disjoint planes, and hence is rational. In fact, we prove that a cubic fourfold with the automorphism \(\phi_3^4\) also admits the automorphism \(\phi_3^3\). The existence of planes in such cubic fourfolds is due to the existence of Eckardt points. By studying the Eckardt points, we can deduce information on the configuration of the corresponding planes.

\begin{definition}
Let \(X\) be a smooth cubic fourfold. We say that \(P\in X\) is an \textit{Eckardt point} if \(P\) has multiplicity \(3\) in \(T_PX\cap X\), equivalently if \(T_PX\cap X\) is a cone with vertex \(P\) over a cubic surface.  
\end{definition}
It is proved in \cite{Laza.Perlstein.Zheng:Eckardt} that, for a cubic fourfold \(X\), having an Eckardt point \(P\in X\) is equivalent to being invariant for the anti-symplectic involution \([x_0:\dots:x_4:x_5]\mapsto [x_0:\dots:x_4:-x_5]\) after a suitable change of coordinates, where the Eckardt point is given by \(P=[0:\dots:0:1]\). We will refer to such an involution as an \textit{Eckardt involution}, this is of the form \(\phi_1\) referring to \cite{marquand2023cubic}.

According to \cite{Laza.Zheng:aut}, for a general cubic fourfold with automorphism \(\phi_3^3\) we have \(\Aut_S(X)=\langle\phi_3^3\rangle\). 
\begin{proposition}\label{eckardt_points_in_phi_3^3}
    A general cubic fourfold \(X\) in the family \(F_3^3\) contains three Eckardt points. It contains exactly \(81\) planes that generate the lattice \(A(X)\). Moreover, \(X\) contains two disjoint planes and it is rational. 
\end{proposition}
\begin{proof}
Recall from \autoref{class_cubiche} that the equation for a general such $X$ is given as:
\[L_3(x_0,\dots,x_3)+x_4^3+x_5^3+x_4x_5M_1(x_0,\dots,x_3)=0.\]
One can see that $X$ admits an involution that exchanges the two variables \(x_4\) and \(x_5\). After the change of basis given by \(y_4=x_4+x_5\) and \(y_5=x_4-x_5\), the involution is realized as \(y_5\mapsto -y_5\) and is an Eckardt involution. In the original coordinates, the involution has an Eckardt point \(P_1=[0:\dots:1:-1]\); we see two other Eckardt points in the orbit of $P_1$ under $\phi_3^3$. More explicitly, they are given by coordinates \(\phi_3^3(P_1)=P_2=[0:\dots:1:-\xi_3]\) and \(\phi_3^3(P_2)=P_3=[0:\dots:1:-\xi_3^2]\). 
   
Each Eckardt point \(P_i\) determines a cone \(X\cap T_{P_i}X=C_i\subset X\) over a cubic surface \(S_i\) for \(i=1,2,3\), moreover one can see that \(P_i\not\in T_{P_j}X\) for \(i\not=j\). This shows the existence of the \(81\) planes in \(X\), each generated by an Eckardt point \(P_i\), \(i=1,2,3\) and one of the \(27\) lines on the cubic surface $S_i$. For \(i\not= j\), we have that \(T_{P_i}X\cap T_{P_j}X\cong \mathbb{P}^3\), this determines a cubic surface \(S_{ij}=X\cap T_{P_i}X\cap T_{P_j}X\) contained in \(C_i\cap C_j\) and not containing the points \(P_i, P_j\).
As a consequence, there are at least two disjoint lines in the cubic surface \(S_{ij}\) that span two disjoint planes, one plane in the cone with vertex \(P_i\) and one in a cone with vertex another Eckardt point \(P_j\), and hence \(X\) is rational (in general, a cubic fourfold with at least two Eckardt points is rational, see \cite[Corollary 6.4.2]{gammelgaard2018cubic}). It is easily checked using the matrix description of \autoref{algebraic_lattices} that there are exactly \(81\) classes \(\alpha\) in \(A(X)\) such that \(\alpha^2=3\) and \(\alpha\cdot \eta_X=1\), moreover these generate the entire lattice (this can be checked using a computer).
\end{proof}
A general cubic fourfolds \(X\) in the family \(F_3^4\) is studied in \cite{koike2022cubic}, where it is proved that it contains \(243\) planes. 
The group \(\Aut_S(X)\) was described in \cite{Laza.Zheng:aut} and explicit generators of \(\Aut(X)\) are given in \cite{koike2022cubic}.

\begin{proposition}\label{eckardt_points_in_phi_3^4}
    A general cubic fourfold \(X\) in the family \(F_3^4\) belongs to the families \(F_2^1, F_2^2, F_3^3\) and \(F_3^6\). It contains eighteen Eckardt points and it contains exactly \(243\) planes that generate the lattice \(A(X)\). Moreover, \(X\) contains disjoint planes and it is rational. 
\end{proposition}

\begin{proof}
    Recall that from \autoref{class_cubiche} the equation of such an $X$ is given by $L_3(x_0,x_1,x_2)+M_3(x_3, x_4, x_5)=0$. The cubic curves determined by \(L_3\) and \(M_3\) can be put in the Hesse normal form: there exist \(\lambda,\mu\in\mathbb{C}\) with \(\lambda^3\not=1\) and \(\mu^3\not=1\) such that the equation of \(X\) is of the form 
\[ x_0^3+x_1^3+x_2^3 -\lambda x_0x_1x_2=x_3^3+x_4^3+x_5^3 -\mu x_3x_4x_5.\] 
One can see that \[\Aut(X)\cong \Aut(L_3)\times\Aut(M_3)\times \langle \phi_3^4\rangle\] 
where \(\Aut(L_3)\) and \(\Aut(M_3)\) are of order \(18\) generated by the permutations of variables and \[L=\begin{pmatrix}
    1 & 0 & 0\\
    0 & \xi_3 &0\\
    0 & 0 & \xi_3^2
\end{pmatrix}.\]
The group of symplectic automorphisms \(\Aut_S(X)\) is the subgroup of index \(2\) given by automorphisms of determinant \(1\) (only even permutations).
From this description, we see that the cubic belongs to the families \(F_2^1, F_2^2, F_3^3\) and \(F_3^6\). 

With the variables of the Hesse normal form, we see that any of the \(9\) points in the \(\Aut(L_3)\)-orbit of \(P=[1:-1:0:\dots:0]\) is an Eckardt point.
More precisely, any point \(P_i\) in the orbit is the vertex of the cone \(C_i=X\cap T_{P_i} X\) over a cubic surface \(S_i\) for \(i \in \{1, \cdots, 9\}\) and one can directly see from the equations that for a general choice of \(\lambda\) we have \(P_j\not\in T_{P_i} X\) whenever \(i\not=j\). Thus, any choice of a line \(l\) in \(S_i\) gives a distinct plane as a cone over \(l\) with vertex the point \(P_i\).

Any of the nine Eckardt points \(P_i\), \(i \in \{1, \cdots, 9\}\) determines \(27\) planes in \(X\), then we find \(243\) planes in \(X\). From the matrix description of the lattice \(A(X)\) in \autoref{algebraic_lattices} one can see that there are exactly \(243\) classes \(\alpha\) such that \(\alpha^2=3\) and \(\alpha\cdot \eta_X=1\). 
Each such class is represented by a unique plane, which generate the entire lattice.
We have that \(T_{P_i}X \cap T_{P_j}X \cong \mathbb{P}^3\) for \(i\not=j\), hence \(S_{ij}=X\cap T_{P_i}X \cap T_{P_j}X\) is a cubic surface contained in the intersection of the cones \(C_i\cap C_j\). Taking the cones over two disjoint lines \(l_i, l_j\in S_{ij}\) with vertices \(P_i\) and \(P_j\) respectively, we find disjoint planes in \(X\) and then \(X\) is rational.

There are \(9\) other Eckardt points in \(X\) corresponding to the \(\Aut(M_3)\)-orbit of the point \(Q=[0:\dots:1:-1:0]\), but they determine the same planes as the previous Eckardt points. In fact, let \(C\) be the cone over a cubic surface associated to a point in the orbit of \(P\) and \(C'\) the cone associated to a point in the orbit of \(Q\). The intersection \(C\cap C'\) consist of \(3\) disinct planes - indeed, after a suitable change of coordinated the equation of \(X\) can be re-written as \[y_0y_1^2+y_2^3= y_3 y_4^2+y_5^3 .\] The cones are given by \(C=X\cap \{y_0=0\}\), \(C'=X\cap \{y_3=0\}\) and the intersection is \(C\cap C'=\{y_2^3=y_5^3\}\subset \mathbb{P}^3\). From the fact that there are \(9\) points in each orbit, it follows that the \(27\) planes associated to an Eckardt point are contained in the \(9\) cones corresponding to Eckardt points in the other orbit.
\end{proof}

\section{G-rationality}\label{G_rationality_section}
One of the guiding questions in the theory of cubic fourfolds is to determine if a cubic fourfold is rational. In the presence of a group action, it is natural to ask whether rationality holds in the equivariant context. This has been recently explored in \cite{Bohning.Bothmer.Tschinkel:Equivariant.birational.geometry}, who provided a counter example to the equivariant analogues of existing rationality conjectures. 
In particular, the authors provide an example of a rational cubic fourfold (a Pfaffian cubic) with a large group action that is not $G$-rational. 
In this section we provide two more examples or rational cubic fourfolds, which are not $G$-rational for $G=\Aut(X)$.
\begin{definition}
Let \(G\) be a group, then a \(G\)\textit{-variety} is a variety \(X\) with an inclusion \(G\subseteq \Aut(X)\). Two \(G\)-varieties are said to be \(G\)\textit{-birational} if there is a \(G\)-equivariant birational morphisms between them. A \(G\)-variety is said to be \(G\)\textit{-rational} if it is \(G\)-birational to \(\mathbb{P}^n\).
\end{definition}

Clearly, if a group \(G\) cannot be realized as a group of linear transformation of \(\mathbb{P}^n\) then no variety of dimension \(n\) can be \(G\)-rational. We now introduce the invariant that will allow us to determine that a cubic fourfold is not \(G\)-rational, following for example \cite{tschinkel2023equivariantbirationalgeometrylinear}. The same definitions could be given for an arbitrary ground field, but for simplicity we stick to the case of complex numbers.

Let \(G\) be a group and \(n\in\mathbb{N}\), denote by \(\Symb_n(G)\) the free abelian group generated by symbols \((H,Y\actsfromleft K,\beta)\)
where 
\begin{itemize}
    \item \(H\subseteq G\) is an abelian subgroup,
    \item \(Y\subseteq Z_G(H)/H\) is a subgroup, where \(Z_G(H)\) denotes the centralizer of \(H\) in \(G\),
    \item \(K\) is a finitely generated field extension of \(\mathbb{C}\) of transcendence degree \(d\subseteq n\) with a faithful action of \(Y\),
    \item \(\beta=(\beta_1,\dots,\beta_{n-d})\) is a sequence of nontrivial characters of \(H\) generating \(H^\vee\).
\end{itemize}

The \textit{equivariant Burnside group} \(\Burn_n(G)\) is obtained as a quotient of \(\Symb_n(G)\) by certain (slightly technical) equivalence relations, which we omit but are precisely described in \cite[Section 3.3]{tschinkel2023equivariantbirationalgeometrylinear}. 
The relation which will be relevant for us,  is the blow-up relation:
\((H,Y\actsfromleft K,\beta)=\Theta_1+\Theta_2\), where 
\begin{align*}
    \Theta_1&:=\begin{centercases}
  0      & \text{if } b_1=b_2 \\
  (H,Y\actsfromleft K,\beta_1)+(H,Y\actsfromleft K,\beta_2) & \text{if }b_1\not=b_2
       \end{centercases}\\ 
    \Theta_2&:=\begin{centercases}
            0      & \text{if }b_i\in\langle b_1-b_2\rangle \text{ for some i}, \\
          (\overline{H},\overline{Y}\actsfromleft K(t),\overline{\beta}) & \text{otherwise.}
       \end{centercases}
\end{align*}
Here, \(\overline{H}=\Ker(b_1-b_2)\subset H\), \(\beta\) is the image of characters in \(\overline{H}^\vee\), \(\beta_1=(b_1-b_2,b_2,b_3,\dots,b_n)\) and \(\beta_2=(b_1,b_2-b_1,b_3,\dots,b_n)\). We remark that this is the only relation that involves extensions with different transcendence degrees.

For \(X\) a \(G\)-variety of dimension \(n\), following  \cite{Kresch2020EquivariantBT,tschinkel2023equivariantbirationalgeometrylinear}, we consider the symbol:
 \[[X\righttoleftarrow G]:=\sum_H\sum_F(H, Y \actsfromleft k(F),\beta_F)\in\Burn_n(G).\]
Here the sum is over the conjugacy classes of stabilizers \(H\) of maximal strata \(F\subset X\) with these stabilizers, induced action of a subgroup \(Y\subseteq Z_G(H)/H\) on the field of functions \(\mathbb{C}(F)\) and corresponding weights \(\beta_F\) of \(H\) in the normal bundle of \(F\). The symbol $[X\righttoleftarrow G]$ is well-defined and is a \(G\)-birational invariant.

\begin{theorem}\label{G_rationality}
    Let \(X\) be a cubic fourfold with automorphism \(\phi_3^3\) or \(\phi_3^4\) and let \(G=\Aut(X)\). Then \(X\) is rational but not \(G\)-rational.
\end{theorem}
\begin{proof}
Suppose that \(X\) is \(G\)-rational, i.e. there is a \(G\)-equivariant birational morphism between \(X\) and \(\mathbb{P}^4\). Then by \cite[Theorem 5.15]{Kresch2020EquivariantBT} we have equality \[[X\righttoleftarrow G]=[\mathbb{P}^4\righttoleftarrow G]\in\Burn_4(G),\]
and by \cite[Corollary 6.1]{tschinkel2023equivariantbirationalgeometrylinear} for every summand in \([\mathbb{P}^4\righttoleftarrow G]\) the stratum \(F\) is birational to a product \(\Pi_j \mathbb{P}(W_j)\) of linear space where the induced action on each factor is birational to a projective linear action.

In any case, there is an involution of \(X\) point-wise fixing an Eckardt point and smooth cubic threefold \(F\subset X\) which gives a symbol \((C_2,Y\actsfromleft k(F),(-1))\in\Burn_4(G)\) appearing as a summand in \([X\righttoleftarrow G]\), we argue that this gives a contradiction to the equality \([X\righttoleftarrow G]=[\mathbb{P}^4\righttoleftarrow G]\). In fact, the cubic threefold \(F\) is irrational, and the associated symbol is different from any symbol appearing in \([\mathbb{P}^4\righttoleftarrow G]\). Moreover, the symbol associated to the cubic threefold is not subject to the blow-up relation with any of the symbols appearing in \([\mathbb{P}^4\righttoleftarrow G]\) because the associated strata are rational, and the blow-up relation preserves rationality.
\end{proof}

\begin{remark}
    One can use the same techniques of \autoref{G_rationality} to show that the general cubic fourfold $X$ admitting an Eckardt involution is not $G$-rational for $G=\mathbb{Z}/2\mathbb{Z}$, as suggested by Tschinkel. However, such a cubic is conjecturely irrational (see \cite{marquand2023cubic} for a discussion).
\end{remark}

\section{Birational transformations of LSV manifolds}\label{LSV_birationalities}
In this section we consider a cubic fourfold \(X\) with a symplectic automorphism of prime order and study the cohomological action of the induced birational transformation on associated LSV manifolds \(J(X)\) and associated twisted LSV manifolds \(J^t(X)\). The classification of groups of birational transformations of an IHS of OG10 type that can be induced from automorphisms of a cubic fourfold was completed in \cite[Theorem 6.1]{marquand2024finitegroupssymplecticbirational} (see also \cite{marquand2023classificationsymplecticbirationalinvolutions} for the case of involutions). Here we are able to identify the Neron-Severi and transcendental lattices of \(J(X)\) and \(J^t(X)\), complimenting the above results.

Recall that an \textit{irreducible holomorphic symplectic (IHS) manifold} is a compact Kähler manifold which is simply connected and whose space of holomorphic two forms is generated by a non-degenerate form (see \cite{Beauville:varietes.kahleriennes.classe.chern.nulle}). One can construct smooth IHS manifolds from a cubic fourfold - we briefly recall the construction of \cite{LSV,Voisin:compactification.twisted.2016hyper,Sacca2020birational}.

Let \(X\subset \mathbb{P}^5\) be a smooth cubic fourfold, and consider the intermediate Jacobian fibration \(\pi_U:J_U(X)\to U\subset (\IP^5)^\vee\) over the locus $U$ of smooth hyperplane sections, whose fibers are intermediate Jacobians.
Similarly, there is a fibration 
\(\pi_U^t:J^t_U(X)\to U\) whose fibers are twisted intermediate Jacobians (parametrizing degree \(1\) cycles). We recall the following result:

\begin{theorem}[{\cite{LSV,Voisin:compactification.twisted.2016hyper,Sacca2020birational}}]  
Let \(X\) be a smooth cubic fourfold, then there exists  compactifications \(J(X)\to (\mathbb{P}^5)^\vee\) and \(J^t(X)\to (\mathbb{P}^5)^\vee\) such that \(J(X)\) and \(J^t(X)\) are irreducible symplectic manifolds of OG10 type.
\end{theorem}

Such a IHS compactification \(J(X)\) is called a \textit{Laza-Saccà-Voisin (LSV)} manifold, and \(J^t(X)\) is called a \textit{twisted Laza-Saccà-Voisin (twisted LSV)}.

There are two algebraic classes which are always present, namely the pullback of the hyperplane class that we denote by \(L\) and \(L^t\) and the relative theta divisor that we denote by \(\Theta\) and \(\Theta^t\) respectively. These classes span a lattice \[\bU_X=\langle L, \Theta\rangle\subseteq \NS(J(X))\] that is isometric to the hyperbolic plane \(\bU\), and a lattice 
\[\bU^t_X=\langle L^t, \Theta^t\rangle\subseteq \NS(J(X))\] that is isometric to the rescaled hyperbolic plane \(\bU(3)\) (see \cite{Mongardi.Onorati:birational.geometry.OG10}). Moreover, there are isometries 
\[\bU_X^{\perp}\cong H^4_{prim}(X,\mathbb{Z})(-1)\cong (\bU_X^t)^{\perp},\]
where the orthogonal complements are taken in \(H^2(J(X),\mathbb{Z})\) and \(H^2(J^t(X),\mathbb{Z})\), respectively.

We firstly establish the relation between the coinvariant lattice for \(X\) and the coinvariant lattices for \(J(X)\) and \(J^t(X)\), this allows us to also deduce a relation between the algebraic lattices when \(X\) is general.
\begin{proposition}\label{inv_to_inv}
Let \(X\) be a cubic fourfold which is general among the ones admitting a symplectic automorphism \(\phi\in\Aut(X)\) of finite order and let \(\widetilde{\phi}\in\Bir(J(X))\) and \(\widetilde{\phi}^t\in\Bir(J^t(X))\) be the induced birational transformations on the manifolds \(J(X)\) and \(J^t(X)\).

Then we have:
\begin{itemize}
    \item an isometry
\(H_{prim}^4(X,\mathbb{Z})_\phi(-1)\cong  H^2(J(X),\mathbb{Z})_{\widetilde{\phi}}\),
\item an isometry
    \(H_{prim}^4(X,\mathbb{Z})_\phi(-1)\cong  H^2(J^t(X),\mathbb{Z})_{\widetilde{\phi}^t},\)
\item 
an equality \(
      \bU_X\oplus  H^2(J(X),\mathbb{Z})_{\widetilde{\phi}}=\NS(J(X)) \), \item  an embedding \(
  \bU_X^t\oplus H^2(J^t(X),\mathbb{Z})_{\widetilde{\phi}^t}\subseteq \NS(J^t(X))
\),
\end{itemize}
where the last embedding is of index \(3\) if \(X\in \mathcal{C}_d\) for some \(d\equiv 2(6)\), and it is primitive (and hence an equality) otherwise. 
\end{proposition}
\begin{proof}
We argue as in the proof of \cite[Proposition 5.6]{billigrossi2024non}: there is an equivariant isogeny
of Hodge structures \[\widetilde{\alpha} \colon H_{prim}^4(X,\mathbb{Z})(-1) \rightarrow \bU_X^\perp\subset H^2(J(X),\mathbb{Z}).\]  It follows that \(H_{prim}^4(X,\mathbb{Z})_{\phi}(-N) \subseteq (\bU_X^{\perp})_{\widetilde{\phi}}= H^2(J(X),\mathbb{Z})_{\widetilde{\phi}}\) for some integer \(N>0\).
Since the cubic fourfold is general, we have \(H_{prim}^4(X,\mathbb{Z})_{\phi} \cong A_{prim}(X)\) and there are finite index embeddings \[ A_{prim}(X)(-N) \cong H_{prim}^4(X,\mathbb{Z})_{\phi}(-N) \subseteq (\bU_X^{\perp})_{\widetilde{\phi}} \subseteq (\bU_X^{\perp})^{1,1},\]
from which we conclude that \( H^2(J(X),\mathbb{Z})_{\widetilde{\phi}}=(\bU_X^{\perp})_{\widetilde{\phi}}= (\bU_X^{\perp})^{1,1}\) by observing that the last embedding is primitive and the lattices have the same rank.
We obtained an embedding \(\bU_X\oplus  H^2(J(X),\mathbb{Z})_{\widetilde{\phi}}\subseteq \NS(J(X))\) and the same argument gives an embedding \(\bU^t_X\oplus  H^2(J^t(X),\mathbb{Z})_{\widetilde{\phi}}\subseteq \NS(J^t(X))\). 
In the first case, the embedding is primitive and hence an equality, whereas in the second case the lattice \(\bU_X^t\cong \bU(3)\) has gluing subgroup \(\mathbb{Z}/3\mathbb{Z}\) with its orthogonal in \(H^2(J^t(X),\mathbb{Z})\). By \autoref{lem:Z3_discrim} and \autoref{lem:d_cong_2_mod_6}, the gluing subgroup \(\mathbb{Z}/3\mathbb{Z}\) lies in \(A_{prim}(X)=H_{prim}^4(X,\mathbb{Z})_{\phi}\) in the case \(X\in\mathcal{C}_d\) for some \(d\equiv 2(6)\), otherwise it lies in \(T(X)\).

The isometry  \(H_{prim}^4(X,\mathbb{Z})_\phi(-1)\cong  H^2(J^t(X),\mathbb{Z})_{\widetilde{\phi}^t}\) directly follows from \cite[Proposition 5.1]{billigrossi2024non} together with the generality assumption \(A_{prim}(X)=H_{prim}^4(X,\mathbb{Z})_\phi\). From the reasoning above, we know that there is an embedding \(H_{prim}^4(X,\mathbb{Z})_{\phi}(-N)\subseteq H^2(J(X),\mathbb{Z})_{\widetilde{\phi}}\). Moreover, the lattices \(H_{prim}^4(X,\mathbb{Z})_{\phi}(-1)\) and \(H^2(J(X),\mathbb{Z})_{\widetilde{\phi}}\) are 
\(p\)-elementary with the same rank and admit a fix-point free isometry of order \(p\). 
We now have two cases: 
\begin{itemize}
    \item A cubic \(X\) fourfold with automorphism \(\phi_3^3\) or \(\phi_3^4\) contains planes and \(X\in\mathcal{C}_8\): this implies that \(J(X)\) is birational to \(J^t(X)\) by \cite[Theorem 1.3]{Li.Pertusi.Zhao:elliptic.quintics} and \cite[Theorem 4.3]{Giovenzana.Giovenzana.Onorati:Li.Pertusi.Zhao}. Hence, there is a Hodge isometry \(H^2(J(X),\mathbb{Z})\cong H^2(J^t(X),\mathbb{Z})\).
    \item 
    For a cubic fourfold with automorphism \(\phi\in \{\phi_2^2, \phi_3^6, \phi_5^1,\phi_7^1,\phi_{11}^1\}\), the conditions in \cite[Proposition 2.16]{Brandhorst.Cattaneo:prime.order.unimodular} show that the inequality \[\rank(H^2(J(X),\mathbb{Z})_{\widetilde{\phi}})+l(H^2(J(X),\mathbb{Z})_{\widetilde{\phi}})\leq 24,\] is satisfied and \cite[Proposition 3.3]{Laza.Zheng:aut} implies that the coinvariant lattice \(H^2(J(X),\mathbb{Z})_{\widetilde{\phi}}\) with the isometry \( \widetilde{\phi}\) form a Leech pair. Leech pairs are classified in \cite{Hoehn.Mason:290.fixed.Leech}, and there is a unique pair with a \(p\)-elementary lattice of given rank.
\end{itemize}
In all cases, we see that when \(\phi\in\Aut(X)\) is symplectic of prime order we have an isometry \(H_{prim}^4(X,\mathbb{Z})_\phi(-1)\cong  H^2(J(X),\mathbb{Z})_{\widetilde{\phi}}\) and this concludes the proof.
\end{proof}
In particular, if \(X\) is general among the cubic fourfolds with a symplectic automorphism \(\phi\) of prime order then we have equalities \[\NS(J(X))=\bU_X\oplus H^2(J(X),\mathbb{Z})_{\widetilde{\phi}}\text{, }\NS(J^t(X))=(\bU_X^t\oplus H^2(J^t(X),\mathbb{Z})_{\widetilde{\phi}^t})_{sat} \]
and 
\[T(J(X))= \bU_X^\perp \subset H^2(J(X),\mathbb{Z})^{\tilde{\phi}} \text{, }T(J^t(X))=(\bU_X^t)^\perp\subset H^2(J^t(X),\mathbb{Z})^{\widetilde{\phi}^t},\]
i.e. \(J(X)\) and \(J^t(X)\) are also general. Here \((-)_{sat}\) denotes the operation of taking the smaller primitive lattice in \(H^2(J^t(X),\mathbb{Z})\) containing the given lattice, namely its saturation. 
As a direct application, together with the classification of the invariant and coinvariant lattices for cubic fourfolds, we determine the algebraic lattices of the associated LSV and twisted LSV manifolds. 
\begin{theorem}\label{induced_action_LSV}
    Let \(X\) be a cubic fourfold which is general among the ones admitting a symplectic automorphism \(\phi\in\Aut(X)\) of prime order, then the Néron-Severi lattice and the transcendental lattice of the manifolds \(J(X)\) and \(J^t(X)\) are in \autoref{tab:order_p_sympl_LSV}.
\end{theorem}
\begin{proof}
    The lattices \(H_{prim}^4(X,\mathbb{Z})_\phi\) and \(H_{prim}^4(X,\mathbb{Z})^\phi\) are classified in \autoref{algebraic_lattices}, while the lattices \(\NS(J(X))\) and \(\NS(J^t(X))\) can be computed using \autoref{inv_to_inv}. We notice that \(\NS\) is unique in its genus in all the cases by \cite[Proposition 1.14.2]{Nikulin:int.sym.bilinear.forms}, and it also determines \(T\) uniquely.
\end{proof}

{\begin{longtable}{lllllll}
\caption{Pairs \((\NS,T)\) for \(J(X),J^t(X)\) with birational action induced by a general cubic fourfold with a symplectic automorphism \(\phi\) of prime order \(p\)}
\label{tab:order_p_sympl_LSV} \\
\toprule
No. & \(\rank(\NS)\) &  \(T\) & \(\NS\) & \(\sign(T)\) & bir. model & \(p\)  \\

		\midrule
		\endfirsthead
		
		\multicolumn{7}{c}%
		{\tablename\ \thetable{}, follows from previous page} \\
	No. & \(\rank(\NS)\) &  \(T\) & \(\NS\) & \(\sign(T)\) & bir. model & \(p\)  \\
		
		\midrule
		\endhead
	
		\multicolumn{7}{c}{Continues on next page} \\
		\endfoot
		
		\bottomrule
		\endlastfoot
\(\phi_2^2\) & \(10\) & \( \bU^{\oplus 2}\oplus\bE_8(-2)\oplus\bA_2(-1)\)& \( \bU\oplus \bE_8(-2)

\)& \((2,12)\) & \(J(X)\)  & \(2\) \\

\(\phi_2^2\) & \(10\) & \( \bU^{\oplus 2}\oplus\bE_8(-2)\oplus\bA_2(-1)\)& \( \bU(3)\oplus \bE_8(-2)

\)& \((2,12)\) & \(J^t(X)\)  & \(2\) \\
  
\(\phi_3^3\) & \(14\) & \( \bU(3)^{\oplus 2}\oplus\bE_6(-1)\)& \( \bU\oplus \bA_2(-1)^{\oplus 6}

\)& \((2,8)\) & \(J(X),J^t(X)\)  & \(3\) \\

\(\phi_3^6\) & \(14\) & \( \bU(3)^{\oplus 2}\oplus\bA_2(-1)^{\oplus 3} \)& \( \bU\oplus \bA_2(-1)^{\oplus 6}

\)& \((2,8)\) & \(J(X)\)  & \(3\) \\

\(\phi_3^6\) & \(14\) & \( \bU(3)^{\oplus 2}\oplus\bA_2(-1)^{\oplus 3} \)& \( \bU(3)\oplus \bA_2(-1)^{\oplus 6}

\)& \((2,8)\) & \(J^t(X)\)  & \(3\) \\

\(\phi^4_3\) & \(20\) & \(\bU(3)^{\oplus 2} \)& \( \bU\oplus \bE_8(-1)\oplus \bA_2(-1)^{\oplus 5}

\)& \((2,2)\) & \(J(X),J^t(X)\)  & \(3\) \\
\hline
\(\phi^1_5\) & \(18\) & \(\bU(5)^{\oplus 2}\oplus \bA_2(-1)\) & \( \bU\oplus  \bA_4(-1)^{\oplus 4}
\)& \((2,4)\) & \(J(X)\)  & \(5\) \\

\(\phi^1_5\) & \(18\) & \(\bU(5)^{\oplus 2}\oplus \bA_2(-1)\) & \( \bU(3)\oplus \bA_4(-1)^{\oplus 4}
\)& \((2,4)\) & \(J^t(X)\)  & \(5\) \\
\hline
\(\phi^1_7\) & \(20\) & \(\bU(7)\oplus \bH_{21}\) & \( \bU\oplus \bA_6(-1)^{\oplus 3}
\)& \((2,2)\) & \(J(X)\)  & \(7\) \\

\(\phi^1_7\) & \(20\) & \(\bU(7)\oplus \bH_{21}\) & \( \bU(3)\oplus \bA_6(-1)^{\oplus 3}
\)& \((2,2)\) & \(J^t(X)\)  & \(7\) \\

\hline
\(\phi^1_{11}\) & \(22\) & \(\bA_2(11)\) & \( \bU\oplus \bA_{10}(-1)^{\oplus 2}
\)& \((2,0)\) & \(J(X)\)  & \(11\) \\

        \(\phi^1_{11}\) & \(22\) & \(\bA_2(11)\) & \( \bU(3)\oplus \bA_{10}(-1)^{\oplus 2}
\)& \((2,0)\) & \(J^t(X)\)  & \(11\) \\

	\midrule
    	\end{longtable}}

\appendix
\section{Gram matrices of the algebraic lattices}\label{appendix_gram_matrices}
In this appendix we describe the lattice \(A(X)\) for a cubic fourfold \(X\) which is general among the ones admitting a symplectic automorphism of prime order. The lattice \(A(X)\) is not unique in its genus and to describe it we need to determine its isometry class; to do that, we give the associated Gram matrix. Also, we provide the coordinates of the class \(\eta_X\), which allows to recover the matrix of \(A_{prim}(X)\) from the one of \(A(X)\). 
It turns out that lattices \(A_{prim}(X)\) are known in literature, and then we refer to them precisely in each subsection.

For the reader's convenience, we collect in \autoref{tab:genera_lattices} the genera of the lattices \(A(X),A_{prim}(X)\) and \(T(X)\), according to Conway--Sloane notation \cite[Chapter 15]{Conway.Sloane:sphere.packings.lattices.groups}. 

\begin{center}
{\begin{longtable}{lllllll}
\caption{Genera of triples \((A(X),A_{prim}(X),T(X))\) for a general cubic fourfold \(X\in F_p^i\) with a symplectic automorphism \(\phi_p^i\) of prime order \(p\).}
\label{tab:genera_lattices} \\
		
		\toprule
	 No. & \(A(X)\) &  \(A_{prim}(X)\)  & \(T(X)\)  \\

		\midrule
		\endfirsthead
		
		\multicolumn{5}{c}%
		{\tablename\ \thetable{}, follows from previous page} \\
	 No. & \(A(X)\) &  \(A_{prim}(X)\)  & \(T(X)\)  \\
		
		\midrule
		\endhead
	
		\multicolumn{5}{c}{Continues on next page} \\
		\endfoot
		
		\bottomrule
		\endlastfoot
\(\phi_3^3\) & \(\I_{13,0}3^{-5}\) & \( \II_{12,0}3^{+6}\) & \(\II_{8,2}3^{+5} \)  \\

\(\phi_3^6\) & \(\I_{13,0}3^{+7}\) & \(\II_{12,0}3^{+6}\) & \(\II_{8,2}3^{-7} \) \\

\(\phi^4_3\) & \(\I_{19,0}3^{+4}\) & \(\II_{18,0} 3^{-5}\) & \(\II_{2,2}3^{+4}\)  \\
\hline
\(\phi^1_5\) & \(\I_{17,0}3^{+1}5^{+4}\) & \(\II_{16,0}5^{+4}\) & \(\II_{4,2}3^{-1}5^{+4}\)  \\
\hline
\(\phi^1_7\) & \(\I_{19,0}3^{+1}7^{-3}\) & \(\II_{18,0} 7^{-3}\) & \(\II_{2,2}3^{-1}7^{3}\)  \\

\hline
\(\phi_{11}^1\) & \(\I_{21,0}3^{+1}11^{+2}\) & \( \II_{20,0}11^{+2}\)   & \( \II_{0,2}3^{-1}11^{+2}\)  \\

	\midrule
    	\end{longtable}}
\end{center}

\subsection{The automorphism {\(\phi_3^3\)}}
We have 
\[A(X_3^3)=
\begin{pmatrix}
    
3& 1& -1 &-1& -1& -1&-1& -1& -1& 0& 0& -1& -1\\
1 &3 &-1& 0& 1& -1& -1& 0 &0 &1& 1& 1& -1\\
-1& -1& 3 &1& 0& 1& 1& 1& 1& 1& 1& 1& 1\\
-1& 0& 1& 3& 0& 1& 1& 1& 1& -1& -1& 1& 1\\
-1& 1& 0& 0& 3& 1& 1& 0& 0& 0& 0& 1& -1\\
-1&-1& 1& 1& 1& 3& 1& 1& 1& -1& -1& 0& 0\\
-1& -1& 1& 1& 1& 1& 3& 1& 1& -1& 0& 0& 0\\
-1& 0& 1& 1& 0& 1& 1& 3& 1& 1& 1& 0& 0\\
-1& 0& 1& 1& 0& 1& 1& 1& 3& 0& 1& 1& 0\\
0& 1& 1& -1& 0 &-1& -1& 1& 0& 4 &2 &1& 0\\
0& 1& 1& -1& 0& -1& 0& 1& 1& 2& 4 &1 &0\\
-1& 1& 1& 1& 1& 0& 0& 0& 1& 1& 1& 3& 0\\
-1& -1& 1& 1& -1& 0& 0& 0& 0& 0& 0& 0& 3
\end{pmatrix}\]
with \(\eta_{X_3^3}=(1,0,0,0,0,0,0,0,0,0,0,0,0)\). The lattice \(A(X_3^3)\) is the only index \(3\) overlattice of \(\langle\eta_{X_3^3}\rangle\oplus A_{prim}(X_3^3)\) not containing vectors of square \(1\) which is obtained as in \autoref{lem:Z3_discrim} and \autoref{lem:d_cong_2_mod_6}. Moreover, \(A_{prim}(X_3^3)\cong A_{prim}(X_3^6)\).
\subsection{The automorphism {\(\phi_3^6\)}}
We have \(A(X_3^6)=\langle\eta_{X_3^6}\rangle\oplus A_{prim}(X_3^6)\) and

\[A_{prim}(X_3^6)=\begin{pmatrix} 4 & 1 &-2 & 2 & 2 & 1 &-2 & 2 &-1&  1&  2 &-2\\1 & 4  &0 & 2 &-1 &-1 & 1& -1 &-2&  2& -1&  1\\-2  &0 & 4 &-1& -2 & 0 & 2 & 0 &-1 & 1 & 0 & 0\\ 2 & 2 &-1 & 4 & 0 &-1 & 0 & 0  &0  &2 & 0 &-1\\2 &-1 &-2 & 0 & 4 & 0 &-2 & 2 & 1&  0 & 2& -2\\1 &-1 & 0 &-1 & 0 & 4 &-2 & 2 & 0& -1 & 0 &-1\\-2 & 1&  2 & 0& -2 &-2 & 4 &-2 & 0&  0 &-1 & 2\\ 2 &-1 & 0 & 0 & 2 & 2& -2 & 4 & 0&  1 & 2 &-2\\-1 &-2& -1 & 0&  1 & 0 & 0 & 0 & 4& -1 &-1 & 0\\1 & 2  &1 & 2  &0 &-1 & 0 & 1 &-1&  4 & 1 &-1\\ 2 &-1 & 0 & 0 & 2 & 0& -1 & 2 &-1&  1 & 4 &-2\\-2 & 1&  0 &-1& -2 &-1&  2 &-2 & 0& -1& -2 & 4\end{pmatrix}.\]

\begin{remark}
     The lattice \(A_{prim}(X_3^6)\cong A_{prim}(X_3^3) \cong K_{12} \cong \Omega_3\), where \(K_{12}\) is the one in \cite{plesken1993constructing} (see also the database \cite{Nebe_website}) and \(\Omega_3\) is the one in \cite[Prop. 4.2]{Garbagnati.Sarti:K3.symplectic.prime}. This lattice is famously known as the Coxeter-Todd lattice and it appeared for the first time in \cite{coxeter1953extreme}.
\end{remark}

\subsection{The auomorphism {\(\phi_3^4\)}} We have 
\[A(X_3^4)=
\tiny\begin{pmatrix}
3& 0& 0& 0& 0& 0& 0& 0& 0& 0& 0& 0& 0& 0& 0& 0& 0& 0& -1\\
0& 4& -2& 1& 0& 2& 1& -2& 0& 2& -2& 1& -1& -1& -2& -2& 1& -1& 2\\
0& -2& 4& 1& -3& -2& -1& 2& 2& -1& 1& 1& -1& -1& 2& 2& 0& 2& 4\\
0& 1& 1& 4& -4& -1& -1& 1& 2& 1& 1& 2& 0& -2& -1& 1& 1& 2& 6\\
0& 0& -3& -4& 8& 2& 0& -2& -2& -2& -1& -1& 0& 1& 1& -2& -2& -2& -8\\
0& 2& -2& -1& 2& 4& 0& -2& -1& 2& -2& -1& -1& 0& 0& -2& -1& -2& -2\\
0& 1& -1& -1& 0& 0& 4& -2& -2& 1& -1& 0& 1& 0& -1& 0 &2& -2& 0\\
0& -2& 2& 1& -2& -2& -2& 4& 1 &-1& 1& 0& 0& 1& 0& 2& 0& 2& 1\\
0& 0& 2& 2 &-2& -1& -2& 1& 4& 0& 0& 1& -2& -1& 1 &0& -1& 2 &4\\
0& 2& -1& 1& -2& 2& 1& -1& 0& 4& -2& 0 &0& 0& -1& -1& 0& -1& 2\\
0& -2& 1& 1& -1& -2& -1& 1& 0& -2& 4& 0& 2& 0& 0& 2& 0& 2 &1\\
0& 1& 1& 2& -1& -1& 0 &0 &1 &0 &0 &4 &0 &-2& 0& 0& 1& 1& 5\\
0& -1& -1& 0& 0& -1& 1 &0& -2& 0& 2& 0& 4& 1& -1& 1& 0 &0 &-1\\
0& -1& -1& -2& 1 &0 &0 &1 &-1& 0& 0& -2 &1 &4 &-1 &0 &-1& -1& -4\\
0& -2& 2& -1& 1& 0& -1& 0& 1& -1& 0& 0& -1& -1& 4& 0& -2& 0& -1\\
0& -2& 2& 1& -2& -2& 0& 2& 0& -1& 2& 0& 1 &0 &0 &4& 1& 2 &2\\
0& 1& 0& 1& -2& -1& 2& 0& -1& 0 &0 &1 &0 &-1& -2& 1& 4& 0& 3\\
0& -1& 2& 2& -2& -2& -2& 2 &2& -1& 2& 1& 0& -1& 0& 2 &0 &4 &4\\
-1 &2 &4& 6 &-8& -2& 0& 1& 4& 2& 1& 5& -1& -4& -1& 2 &3 &4 &15
\end{pmatrix}
\]
with \(\eta_{X_3^4}=(1,0,0,0,0,0,0,0,0,0,0,0,0,0,0,0,0,0,0)\).
The lattice \(A(X_4^3)\) is the only index \(3\) overlattice of \(\langle\eta_{X_4^3}\rangle\oplus A_{prim}(X_4^3)\) not containing vectors of square \(1\) which is obtained as in \autoref{lem:Z3_discrim} and \autoref{lem:d_cong_2_mod_6}.
\begin{remark}
 In this case the primitive algebraic lattice \(A_{prim}(X_3^4) \cong \langle \eta_{X_3^4} \rangle ^{\perp} \subset A(X)\) is isometric to the lattice \(K_{18}\) and a matrix for this lattice is available at \cite{plesken1993constructing} (see also the database \cite{Nebe_website}).   
\end{remark}

\subsection{The automorphism {\(\phi_5^1\)}}
We have \(A(X_5^1)=\langle\eta_{X_5^1}\rangle\oplus A_{prim}(X_5^1)\) and
\[A_{prim}(X_5^1)=
\tiny\begin{pmatrix}
4& -2&  0&  1&  2&  0&  2&  2& -2& -1&  0& -2& -2& -1& -2& -2\\ 
-2&  4&  1& -1&  0&  1& -2& -2&  0& -1& -1&  2&  2& -1&  0&  0\\
0&  1&  4&  1&  0&  2& -1& -1&  1&  1&  1& -1&  1&  1& -1& -1\\
1 &-1&  1&  4&  1&  2&  1& -1&  1&  1& -1&  0& -2& -1& -2&  1\\ 2&  0&  0&  1&  4&  1&  2&  1& -2&  0& -2& -1& -2& -2& -1& -1\\
0&  1&  2&  2&  1&  4&  0& -2&  0&  0&  0&  0&  0& -1& -2&  0\\
2& -2& -1&  1&  2&  0&  4&  2& -2&  1&  0& -1& -2&  0& -1& -1\\
2& -2& -1& -1&  1& -2&  2&  4& -2&  0&  0& -2& -1&  1&  0& -2\\ 
-2&  0 & 1&  1& -2&  0& -2& -2&  4&  1&  0&  1&  1&  1&  1&  2\\
-1& -1&  1&  1&  0&  0&  1&  0&  1&  4&  0& -1&  0&  1&  1&  1\\ 
0 &-1&  1& -1& -2&  0&  0&  0&  0&  0&  4& -1&  1&  2&  0& -1\\
-2&  2& -1&  0& -1&  0& -1& -2&  1& -1& -1&  4&  1& -1&  0&  1\\
-2&  2&  1& -2& -2&  0& -2& -1&  1&  0&  1&  1&  4&  1&  1&  0\\
-1& -1&  1& -1& -2& -1&  0
&1 & 1 & 1 & 2 &-1 & 1 & 4 & 1 & 0\\
-2&  0& -1& -2& -1& -2& -1&  0&  1&  1&  0&  0&  1&  1&  4&  1\\
-2&  0& -1&  1& -1&  0& -1& -2&  2&  1& -1&  1&  0&  0&  1&  4
\end{pmatrix}.
\]
\begin{remark}
    The lattice \(A_{prim}(X_5^1)\cong \Omega_5 \cong DIH_{16}(10)\), where \(\Omega_5\) is the lattice described in \cite[Prop. 4.4]{Garbagnati.Sarti:K3.symplectic.prime} and \(DIH_{16}(10)\) in \cite{EE8_lattices}. 
\end{remark}

\subsection{The automorphism {\(\phi_7^1\)}}
We have \(A(X_7^1)=\langle\eta_{X_7^1}\rangle\oplus A_{prim}(X_7^1)\) and
\[A_{prim}(X_7^1)=\tiny\begin{pmatrix}
     4& -1& -1&  1&  1& -1& -1&  2& -1& -1&  2& -2&  2&  2& -1& -1&  1&  2\\
     -1&  4&  1& -2&  1&  2&  2& -2&  1&  2&  1& -1&  0&  1&  1&  2& -2&  0\\
     -1&  1&  4& -1& -1&  2& -1& -1&  2& -1& -1&  0& -2&  0&  2&  2&  1&  1\\
     1 &-2& -1&  4& -1& -2& -2&  2& -2& -2&  1&  1&  0&  1& -2& -1&  0&  0\\
     1&  1& -1& -1&  4&  1&  2& -1&  1&  1&  1& -2&  0&  1& -1&  1&  0&  1\\
     -1&  2&  2& -2&  1&  4&  1& -1&  2&  0& -1&  0& -2&  1&  2&  1& -1&  0\\
     -1&  2& -1& -2&  2 & 1&  4& -2&  0&  2&  0&  0&  0&  0&  0&  1& -2& -1\\
     2& -2& -1&  2& -1& -1& -2&  4& -2& -2&  0&  0&  1&  1&  0& -2&  0&  1\\
     -1&  1&  2& -2&  1&  2&  0& -2&  4&  1& -1& -1& -2&  0&  1&  1&  1&  1\\
     -1&  2& -1& -2&  1&  0&  2& -2&  1&  4&  0& -1&  1& -1&  0&  1& -1&  0\\
     2&  1& -1&  1&  1& -1&  0&  0& -1&  0&  4& -2&  2&  2& -2&  0&  0&  0\\
     -2& -1&  0&  1& -2&  0&  0&  0& -1& -1& -2&  4& -1& -1&  0& -1& -1& -2\\
     2&  0& -2&  0&  0& -2&  0&  1& -2&  1 & 2 &-1 & 4&  0& -1& -1&  0 & 0\\
     2&  1&  0&  1&  1&  1&  0&  1&  0& -1&  2& -1&  0&  4& -1&  0& -1&  1\\
     -1&  1&  2& -2& -1&  2&  0&  0&  1&  0& -2&  0& -1& -1&  4&  1&  0&  0\\
     -1&  2&  2& -1&  1&  1&  1& -2&  1&  1&  0& -1& -1&  0&  1&  4&  0&  1\\
     1& -2&  1&  0&  0& -1& -2&  0&  1& -1&  0& -1&  0& -1&  0&  0&  4&  1\\  2&  0&  1&  0&  1&  0& -1&  1&  1&  0&  0& -2&  0&  1&  0&  1&  1&  4\end{pmatrix}.\]
     
\begin{remark}
     We have \(A_{prim}(X_7^1)\cong \Omega_7\), where \(\Omega_7\) is the lattice described in \cite[Prop. 4.6]{Garbagnati.Sarti:K3.symplectic.prime}.
\end{remark}

\subsection{The automorphism {\(\phi_{11}^1\)}}

We have \(A(X_{11}^1)=\langle\eta_{X_{11}^1}\rangle\oplus A_{prim}(X_{11}^1)\) and
\[A_{prim}(X_{11}^1)=\tiny
\begin{pmatrix}
 4& -1&  2& -1&  0& -1&  2&  2&  1&  2&  2& -1&  2& -1&  2&  2& -2&  2& -1& -2\\
 -1&  4&  0&  0&  0& -1& -1& -2& -2& -2& -2& -1& -2&  2& -2& -1&  2&  0& -1&  1\\
 2&  0&  4& -1&  1&  0&  2 & 0&  1&  1&  2& -2&  1& -1&  0&  1& -1&  2& -1& -1\\
 -1&  0& -1&  4&  1&  1&  0&  0& -2& -2& -1&  1&  1& -1& -2& -2&  2& -1&  2&  2\\
 0&  0&  1&  1&  4&  2&  2&  1&  1& -1&  0& -1&  1& -1 & 0& -2&  0& -1&  2&  2\\
 -1& -1&  0&  1&  2&  4&  0&  0&  1&  0&  0&  1&  0&  0& -1& -2& -1&  0&  1&  1\\
 2& -1&  2&  0&  2&  0&  4&  2&  1&  1&  2& -1&  2& -1&  1&  0& -1&  1&  0&  0\\
 2 &-2 & 0 & 0 & 1 & 0 & 2 & 4 & 2 & 1 & 2 & 0&  2& -2&  2&  1& -1&  0&  1& -1\\
 1& -2&  1& -2&  1&  1&  1&  2&  4&  2&  2&  0&  1& -1&  2&  1& -2&  0&  0& -1\\
 2 &-2 & 1 &-2& -1&  0&  1&  1&  2&  4&  2&  0&  1&  0&  2&  2& -2&  1& -1& -2\\
 2& -2&  2& -1&  0&  0&  2&  2&  2&  2&  4& -1&  2& -2&  1&  2& -2&  2&  0& -2\\
 -1& -1& -2&  1& -1&  1& -1&  0&  0&  0& -1&  4& -1&  1&  0&  0&  0& -1&  0&  1\\
 2& -2&  1&  1&  1&  0&  2&  2&  1&  1&  2& -1&  4& -2&  1&  0& -1&  1&  1& -1\\
 -1&  2& -1& -1& -1&  0& -1& -2& -1&  0& -2&  1& -2&  4& -1& -1&  0&  0& -2&  0\\
 2 &-2 & 0 &-2&  0& -1&  1&  2&  2&  2&  1&  0&  1& -1&  4&  2& -2&  0&  0& -1\\
 2 &-1&  1& -2& -2& -2&  0&  1&  1&  2&  2&  0&  0& -1&  2&  4& -1&  1& -1& -2\\
 -2&  2& -1&  2&  0& -1& -1& -1& -2& -2& -2&  0& -1&  0& -2& -1&  4& -2&  1&  2\\ 
 2 & 0 & 2& -1& -1&  0&  1&  0&  0&  1&  2& -1&  1&  0&  0&  1& -2&  4& -2& -2\\
 -1& -1& -1&  2&  2&  1&  0&  1&  0& -1&  0&  0&  1& -2&  0& -1&  1& -2&  4&  2\\
 -2&  1& -1&  2&  2&  1&  0& -1& -1& -2& -2&  1& -1&  0& -1& -2&  2& -2&  2&  4
\end{pmatrix}.\]

\begin{remark}
    The lattice \(A_{prim}(X_{11}^1)\cong S_{11,K3^{[2]}} \cong A_{11}\otimes^{(3)}A_2\), where \(S_{11,K3^{[2]}}\) is the lattice defined in \cite[Example 2.9]{Mong13} and \(A_{11}\otimes^{(3)}A_2\) is defined in \cite{nebe1998some}. This lattice is the Mordell-Weil lattice of an elliptic curve over a field of characteristic \(11\) \cite{shioda1991mordell} (see also \cite{Nebe_website}).
\end{remark}
\bibliographystyle{alpha}
\bibliography{references}

\end{document}